\documentclass[a4paper,fleqn,longmktitle]{cas-sc}

\usepackage[numbers]{natbib}

\usepackage{graphicx}%
\usepackage{multirow}%
\usepackage{amsmath,amssymb,amsfonts}%
\usepackage{amsthm}%
\usepackage{mathrsfs}%
\usepackage[title]{appendix}%
\usepackage{xcolor}%
\usepackage{textcomp}%
\usepackage{manyfoot}%
\usepackage{cleveref}
\usepackage{booktabs}%
\usepackage{algorithm}%
\usepackage{algorithmicx}%
\usepackage{algpseudocode}%
\usepackage{listings}%
\usepackage{enumitem}
\usepackage{float}

\theoremstyle{plain}
\newtheorem{theorem}{Theorem}[section]
\newtheorem{lemma}[theorem]{Lemma}
\newtheorem{proposition}[theorem]{Proposition}
\newtheorem{corollary}[theorem]{Corollary}

\theoremstyle{definition}
\newtheorem{assumption}[theorem]{Assumption}
\newtheorem{hypothesis}[theorem]{Hypothesis}

\theoremstyle{remark}
\newtheorem{remark}[theorem]{Remark}

\DeclareMathOperator{\Res}{Res}
\DeclareMathOperator{\Ran}{Ran}
\DeclareMathOperator{\ind}{ind}
\DeclareMathOperator{\spn}{span}
\DeclareMathOperator{\Ai}{Ai}
\DeclareMathOperator{\Bi}{Bi}
\DeclareMathOperator{\Gi}{Gi}
\DeclareMathOperator{\fp}{f.p.}

\newcommand{\Lout}{\mathcal{L}_{\mathrm{out}}}
\newcommand{\Tsh}{\mathcal{T}_{\mathrm{sh}}}
\newcommand{\Gcal}{\mathcal{G}}
\newcommand{\Aout}{\mathcal{A}_{\mathrm{out}}}
\newcommand{\Fout}{\mathcal{F}_{\mathrm{out}}}
\newcommand{\Bcal}{\mathcal{B}}
\newcommand{\Dcal}{\mathcal{D}}
\newcommand{\Ncal}{\mathcal{N}}
\newcommand{\Mcal}{\mathcal{M}}
\newcommand{\Kcal}{\mathcal{K}}
\newcommand{\Hcal}{\mathcal{H}}
\newcommand{\Xcal}{\mathcal{X}}

\newcommand{\Eedge}{\mathcal{E}_{\mathrm{edge}}}
\newcommand{\Redge}{\mathcal{R}_{\mathrm{edge}}}
\newcommand{\LTD}{\mathcal{L}_{\mathrm{TD}}}
\newcommand{\Min}{\mathcal{M}_{\mathrm{in}}}
\newcommand{\Arec}{A_{\mathrm{rec}}}
\newcommand{\Dmatch}{\mathcal{D}_{\mathrm{match}}}
\newcommand{\Bedge}{\mathscr{B}_{\mathrm{edge}}}
\newcommand{\Finc}{\mathbf{F}_{\mathrm{inc}}}
\newcommand{\FKH}{\mathbf{F}_{KH}}
\newcommand{\Fsing}{\mathbf{F}_{\mathrm{sing}}}
\newcommand{\Freg}{\mathbf{F}_{\mathrm{reg}}}
\newcommand{\Vm}{\mathbf{V}_{-}}
\newcommand{\Vz}{\mathbf{V}_{0}}
\newcommand{\Tssing}{\operatorname{Tr}_{\mathrm{sing}}}

\numberwithin{equation}{section}

\def\tsc#1{\csdef{#1}{\textsc{\lowercase{#1}}\xspace}}
\tsc{WGM}
\tsc{QE}

\ExplSyntaxOn
\cs_if_exist:NF \vbox_unpack_clear:N
  {
    \cs_new_protected:Npn \vbox_unpack_clear:N #1
      {
        \vbox_unpack:N #1
        \box_clear:N #1
      }
  }
\ExplSyntaxOff

\makeatletter

\def\@journal{Elsevier}
\makeatother

\usepackage{fancyhdr}
\AtBeginDocument{
    \fancyfoot[L]{J. Yu and L. S. Wang: Preprint submitted to Journal of Mathematical Analysis and Applications}
}
\begin{document}
\let\WriteBookmarks\relax
\def\floatpagepagefraction{0.6}
\def\textpagefraction{.1}

\shorttitle{Fredholm--residue selection of the unsteady Kutta amplitude}  

\shortauthors{J. Yu and L. S. Wang}  

\title [mode = title]{Fredholm--residue selection of the unsteady Kutta amplitude}  

\author[1]{Jiguang Yu}
\fnmark[1] 
\ead{jyu678@bu.edu}

\affiliation[1]{organization={College of Engineering, Boston University},
            city={Boston},
            postcode={02215},
            state={MA},
            country={United States}}

\author[2]{Louis Shuo Wang}
\cormark[1]  
\fnmark[1]  
\ead{wang.s41@northeastern.edu}

\affiliation[2]{organization={Department of Mathematics, Northeastern University},
            city={Boston},
            postcode={02115},
            state={MA},
            country={United States}}

\cortext[1]{Corresponding author}

\fntext[1]{These authors contributed equally to this work as co-first authors.}

\begin{abstract}
We give an operator-theoretic interpretation of unsteady Kutta selection in
trailing-edge acoustic receptivity.  The inviscid acoustic--wake problem leaves
one outgoing wake amplitude undetermined.  We show that, under explicit
structural hypotheses, this amplitude is the same scalar obtained from three
representations: cancellation of the inverse-square-root edge singularity,
Fredholm compatibility of the viscous lower-deck problem, and the residue of
the Kutta-normalized transform solution at the downstream wake pole:
$\displaystyle A = -\frac{C_-^{(0)}}{C_-^{(KH)}} = -\frac{\langle \mathbf F_{\rm inc},\Psi^\ast\rangle}{\langle \mathbf F_{KH},\Psi^\ast\rangle} = i\operatorname*{Res}_{\alpha=\alpha_{KH}}\mathcal M(\alpha)$.
The inner Fredholm--edge mechanism is verified exactly in a linear-shear
lower-deck model, where the primal shear and adjoint velocity are Airy fields
and the edge concomitant is nonzero outside a discrete resonance set.
\end{abstract}


\maketitle

\section{Introduction and main identity}
\label{sec:intro-main}

\begin{figure}
    \centering
    \includegraphics[width=\linewidth]{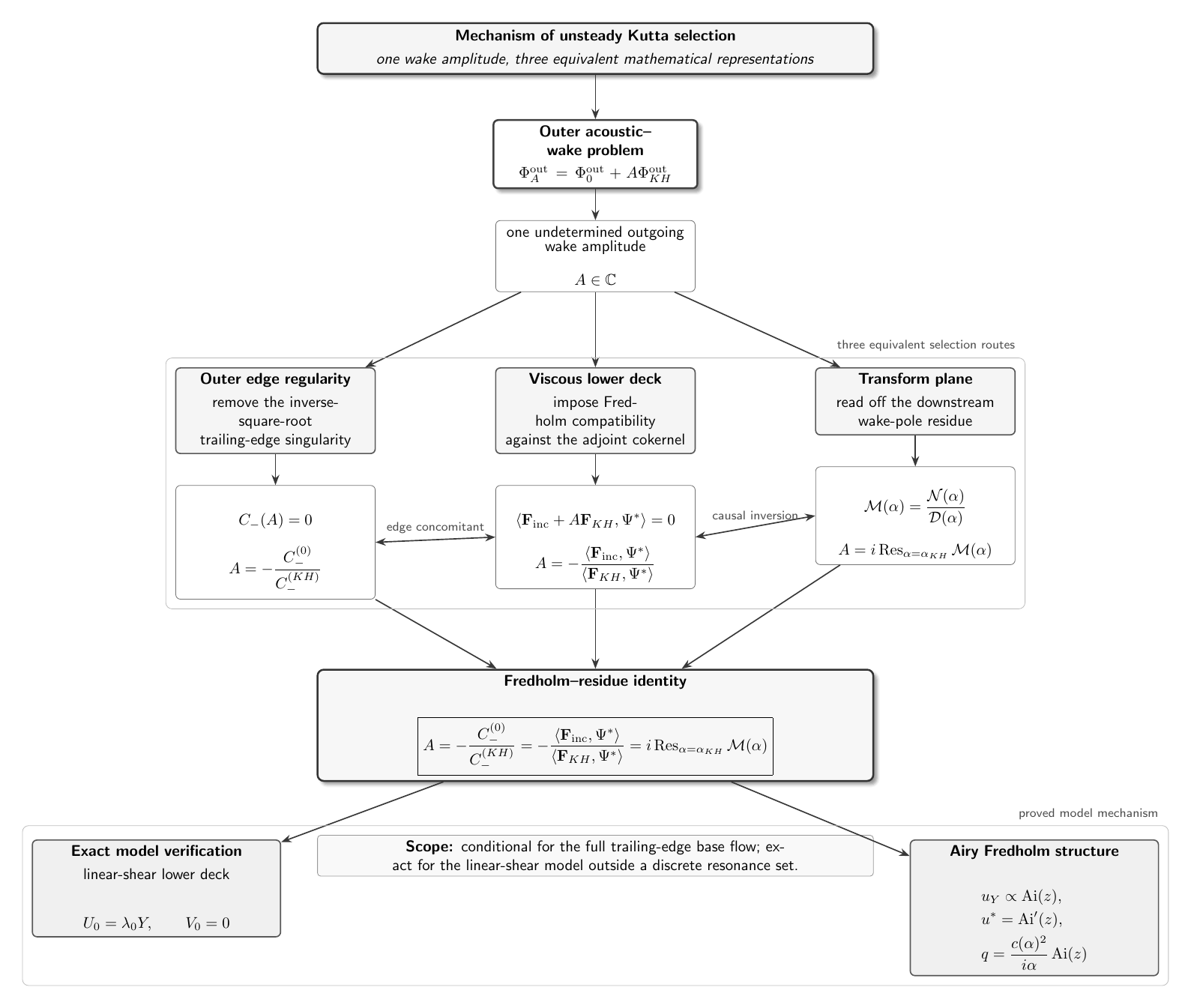}
    \caption{Mechanism of unsteady Kutta selection.  The inviscid acoustic--wake
 problem leaves one outgoing wake amplitude \(A\) undetermined.  The paper
 identifies three equivalent ways to select it: removal of the outer
 inverse-square-root edge singularity, Fredholm compatibility of the viscous
 lower-deck problem, and the residue of the Kutta-normalized transform solution
 at the downstream wake pole.  The inner Fredholm--edge mechanism is verified
 exactly in the linear-shear lower-deck model, where the primal and adjoint
 fields are Airy fields.}
 \label{fig:mechanism-flowchart}
\end{figure}

The unsteady Kutta condition is a singular selection principle at a sharp
edge.  In steady inviscid airfoil theory it fixes the circulation by removing
the inverse-square-root velocity singularity at the trailing edge.  In unsteady
acoustic receptivity the situation is subtler: the outer acoustic--wake problem
may admit an outgoing hydrodynamic wake mode, and the inviscid equations alone
then leave one complex amplitude undetermined.  This paper formulates that
missing scalar as a Fredholm compatibility condition for the viscous
lower-deck problem and identifies the same scalar with the pole residue of a
Kutta-normalized transform solution.  The structural hypotheses under which
the identification holds are isolated and then verified in closed form for the
canonical linear-shear model of the unsteady lower deck, for which the adjoint
state is an Airy-derivative field and the edge concomitant is computable.

That a Kutta condition in unsteady flow need not coincide with its steady
form, and that its applicability is itself a question, was surveyed in \cite{crighton1985kutta,taha2019viscous,wang2025analysis,zhu2020unsteady,ayton2016importance}.  The vortex-sheet-from-an-edge problem and its
sensitivity to the choice of edge condition were analyzed in \cite{orszag1970instability,wang2025analysis1,xia2017unsteady,jones1972instability,crighton1972radiation,crighton1974radiation,liu2025bidirectional,rienstra1981sound,davis2016instability} and the monograph of Howe~\cite{howe1998acoustics,liang2025global}.  The
viscous justification of the steady Kutta condition through triple-deck theory
originates in \cite{yu2026from,stewartson1969flow,messiter1970boundary,stewartson1974multistructured,wang2026algebraic,sychev1998asymptotic,peake1994viscous,jobe1974numerical,wang2026damage}; the unsteady and oscillating-edge
viscous structure was studied by Brown \& Daniels~\cite{brown1975viscous},
Daniels~\cite{daniels1975flow,daniels1978unsteady}, and Brown \&
Stewartson~\cite{brown1975wake}.  Boundary-layer receptivity to sound,
in which an inviscid amplitude is fixed by an edge or solvability mechanism, is
reviewed in \cite{goldstein1989boundary,yu2026pattern,saric2002boundary,goldstein1983evolution,goldstein1985scattering,ruban1984generation,wang2026breakdown,zhong2012direct,zuccher2014boundary,wu2001local,yu2026rigorous}. 
The linearized unsteady lower deck about a uniform shear was solved exactly by
Terent'ev in the vibrating-ribbon problem~\cite{terent1981linear}; the worked
example of Section~\ref{sec:model} is its trailing-edge (plate--wake switching)
analogue.  The downstream/upstream classification of spatial poles we invoke
is the Briggs--Bers criterion~\cite{briggs1964electron,bers1983space,monkewitz1990local,yu2026beyond,hung2015absolute,chen2005convective,antoulinakis2018absolute,cai2026optimal,li2024linear,king2022critical,xie2025local}.
The transform analysis rests on the Wiener--Hopf
technique~\cite{noble1962methods,wang2025multi}; for finite-angle edges it is replaced by Mellin and
functional-difference methods for
wedges~\cite{kondrat1967boundary,lawrie1994exact,davis2016instability,gao2022rolling,wang2026elliptic}.  The
contribution here is to tie the inner (viscous, Fredholm) and outer
(transform, residue) selections together into a single conditional identity,
with all structural hypotheses isolated, and to exhibit a model in which the
inner hypotheses are theorems.
The main contribution is not a full viscous proof for the physical
trailing-edge base flow.  Rather, it is a closed applied-mathematical
selection mechanism: under explicitly stated structural hypotheses, the
unsteady Kutta amplitude is the same scalar in three representations, namely
the outer edge-regularity quotient, the lower-deck adjoint Fredholm quotient,
and the downstream pole residue (\Cref{thm:intro-main}, \Cref{fig:mechanism}).
The inner part of this mechanism is then verified in closed form for the
canonical linear-shear lower-deck model (\Cref{thm:model-anchor}).

Let
\[
        \Gamma_p=(-\infty,0)\times\{0\},\qquad
        \Gamma_w=(0,\infty)\times\{0\},\qquad
        \Pi=\mathbb R^2\setminus(\Gamma_p\cup\Gamma_w),
\]
and let \(O=(0,0)\).  We use \(e^{-i\omega_{\rm phys}t}\), set
\[
        M=\frac Uc\in(0,1),\qquad
        \beta=(1-M^2)^{1/2},\qquad
        k_0=\frac{\omega_{\rm phys}}c,\qquad
        k=\frac{\omega_{\rm phys}L}{U},
\]
and write \(\phi=\phi^{\rm inc}+\phi^{\rm sc}\).  Two length scales are in
play and must be kept distinct.  The outer acoustic--wake problem
\eqref{eq:intro-Lout}--\eqref{eq:intro-Aout} is posed on the hydrodynamic
length \(\ell_\omega=U/\omega_{\rm phys}\), on which the Helmholtz number
\(k_0\ell_\omega=M\) is \(\mathcal O(1)\) uniformly in the Reynolds number; the chord
\(L\) enters only through \(Re\) and the triple-deck scalings of
\Cref{sec:inner}.  All outer coordinates \((x,y)\) below are measured on
\(\ell_\omega\); the intermediate matching between this outer field and the
lower deck is part of the matching map \(\Min\) introduced in
\eqref{eq:forcing-affine} (for the hierarchy of unsteady regions see
\cite{brown1975viscous,daniels1978unsteady}).  In the outer region,
\begin{equation}
\label{eq:intro-Lout}
        \Lout\phi=0,\qquad
        \Lout:=\beta^2\partial_x^2+\partial_y^2
        +2iMk_0\partial_x+k_0^2 .
\end{equation}
On \(\Gamma_p\), \(\gamma_p^\pm\partial_y\phi=0\).  On \(\Gamma_w\), instead
of imposing a degenerate equal-speed convected-sheet closure, we use an
abstract inviscid-sheet transmission law
\[
        \Tsh(\omega_{\rm phys},\Gcal)
        \begin{pmatrix}
        \gamma_w^+\phi\\
        \gamma_w^-\phi\\
        \gamma_w^+\partial_y\phi\\
        \gamma_w^-\partial_y\phi\\
        \eta
        \end{pmatrix}=0,\qquad x>0 .
\]
Here \(\eta(x)e^{-i\omega_{\rm phys}t}\) is the sheet displacement and
\(\Gcal\) denotes the edge geometry, acoustic incidence, and base sheet data.
The equal-speed relations
\[
        \partial_y\phi^\pm=(-i\omega_{\rm phys}+U\partial_x)\eta,\qquad
        (-i\omega_{\rm phys}+U\partial_x)[\phi]=0
\]
are only a neutral limiting model; the downstream instability below belongs to
the full operator \(\Tsh\).
We write the outgoing outer problem as
\begin{equation}
\label{eq:intro-Aout}
        \Aout(\omega_{\rm phys},\Gcal)\Phi=\Fout^{\rm inc},\qquad
        \Phi=(\phi,\eta).
\end{equation}
For a homogeneous normal mode
$\Phi(x,y)=e^{i\alpha x}
        \big(\varphi^+(y),\varphi^-(y),\eta_0\big)$,
one obtains
\[
        (\varphi^\pm)''-\mu(\alpha)^2\varphi^\pm=0,\qquad
        \mu(\alpha)^2=\beta^2\alpha^2+2Mk_0\alpha-k_0^2 ,
\]
with the outgoing branch of \(\mu\).  The sheet law reduces the normal-mode
problem to
\[
        \Bcal(\alpha;\omega_{\rm phys},\Gcal)\mathbf a=0,\qquad
        \Dcal(\alpha;\omega_{\rm phys},\Gcal):=
        \det\Bcal(\alpha;\omega_{\rm phys},\Gcal).
\]
We assume a simple downstream wake pole
\[
        \Dcal(\alpha_{KH};\omega_{\rm phys},\Gcal)=0,\qquad
        \partial_\alpha\Dcal(\alpha_{KH};\omega_{\rm phys},\Gcal)\neq0,\qquad
        \Im\alpha_{KH}<0 .
\]
With \(e^{i\alpha x-i\omega_{\rm phys}t}\), the last inequality corresponds
to downstream spatial growth (a convective spatial instability; see the causal
deformation of \Cref{subsec:flat-residue}).  The associated outgoing
homogeneous field is
$\displaystyle \Phi_{KH}^{\rm out}=(\phi_{KH}^{\rm out},\eta_{KH})$,
normalized once and for all by the wake functional \(\ell_{KH}\) of
\eqref{eq:ellKH-def}, \(\ell_{KH}(\Phi_{KH}^{\rm out})=1\).
The first structural hypothesis is the one-dimensional kernel
$\displaystyle \ker\Aout^{\rm hom}
        =
        \spn\{\Phi_{KH}^{\rm out}\}$.
Thus every outgoing forced outer field has the affine form
\begin{equation}
\label{eq:intro-outer-family}
        \Phi_A^{\rm out}
        =
        \Phi_0^{\rm out}+A\Phi_{KH}^{\rm out},\qquad A\in\mathbb C .
\end{equation}
The unknown scalar \(A\) is the receptivity amplitude.
Near \(O\), let \((r,\theta)\) denote polar coordinates after the local stretching \((x,y)\mapsto(x/\beta,y)\),
\(-\pi<\theta<\pi\).  We assume the standard slit-plane edge pencil has first
nonconstant indicial root \(\lambda=1/2\).  Hence
\[
        \nabla\phi_A^{\rm out}
        =
        C_-(A)r^{-1/2}\Vm(\theta)+\mathcal O(1),
        \qquad
        \Vm(\theta)=\frac12\Psi_-(\theta)e_r+\Psi_-'(\theta)e_\theta,
        \qquad
        \Psi_-(\theta)=\sin\frac{\theta}{2}.
\]
By linearity,
$C_-(A)=C_-^{(0)}+A C_-^{(KH)}$.
The outer regularity form of the Kutta condition is
$C_-(A)=0$.
If \(C_-^{(KH)}\neq0\), this condition alone gives
$\displaystyle A=-\frac{C_-^{(0)}}{C_-^{(KH)}}$.
The central question is why this inviscid regularity condition is selected by
the viscous trailing-edge structure.
For the lower-deck scaling, set
\[
        Re=\frac{UL}{\nu},\qquad
        \varepsilon=Re^{-1/8},\qquad
        x=\varepsilon^3L\,X,\qquad
        y=\varepsilon^5L\,Y,\qquad
        T=\frac{Ut}{\varepsilon^2L}.
\]
Then
\[
        e^{-i\omega_{\rm phys}t}=e^{-i\Omega T},\qquad
        \Omega=\frac{\omega_{\rm phys}\varepsilon^2L}{U}
        =
        \varepsilon^2k
        =
        Re^{-1/4}k .
\]
Thus the distinguished unsteady triple-deck regime is
$\Omega=\mathcal O(1)$ with $k=\mathcal O(Re^{1/4})$.
Linearization of the unsteady lower deck about a steady base state gives
$\LTD(\Omega)W=F$, where $W=(u,v,p,a)^{\mathsf T}$.
The outer-to-inner matching data generated by \eqref{eq:intro-outer-family}
split as
\[
        F_{\rm match}(A)
        =
        C_-(A)\Fsing+\Finc(\Omega)+A\FKH(\Omega),
        \qquad
        \Fsing\in\Eedge,\quad
        \Finc+A\FKH\in\Hcal_\sigma,
\]
where
$\Eedge=\spn\{r^{-1/2}\Vm\}$
is the singular edge-trace space; its lower-deck realization is an
\(|X|^{-1/2}\) line datum in the matching and pressure--displacement
components, and as such it does not belong to the regular data space
\(\Hcal_\sigma\) (\Cref{lem:trace-exclusion}).  Solvability in the bounded
graph domain therefore requires
\[
        \Pi_{\rm sing}F_{\rm match}(A)=0
        \quad\Longleftrightarrow\quad
        C_-(A)=0 .
\]
The Fredholm realization is a closed densely defined operator
$\displaystyle         \LTD(\Omega):\Xcal_\sigma\to\Hcal_\sigma$
between the weighted spaces of \Cref{subsec:weighted-Fredholm}, whose
downstream weight is chosen so that the inner counterpart of the shed wake
mode belongs to the domain (Remark~\ref{rem:kernel-interpretation}).  We assume
\[
        \ind\LTD(\Omega)=0,\qquad
        \ker\LTD(\Omega)^\ast=\spn\{\Psi^\ast(\Omega)\},\qquad
        \langle\FKH(\Omega),\Psi^\ast(\Omega)\rangle_{\Hcal_\sigma}\neq0.
\]
The regular Fredholm condition is
\[
        \Finc(\Omega)+A\FKH(\Omega)\in\Ran\LTD(\Omega)
        \quad\Longleftrightarrow\quad
        \langle\Finc(\Omega)+A\FKH(\Omega),\Psi^\ast(\Omega)\rangle_{\Hcal_\sigma}=0 .
\]
The edge trace and the regular Fredholm projection are linked by the
consistency relation
\begin{equation}
\label{eq:intro-KF-consistency}
        \langle\Finc(\Omega)+A\FKH(\Omega),\Psi^\ast(\Omega)\rangle_{\Hcal_\sigma}
        =
        \chi(\Omega)C_-(A),\qquad
        \chi(\Omega)\neq0 .
\end{equation}
This is not an independent hypothesis: in \Cref{sec:inner}
(Proposition~\ref{prop:KF-noncirc}, Eq.~\eqref{eq:compat-expanded}) it is derived from the
Green identity of Appendix~\ref{app:formal-adjoint} together with the augmented
solvability of {\rm(H3)} below, with \(\chi(\Omega)=-\kappa(\Omega)\); it thus
reduces to nondegeneracy of the edge concomitant, \(\kappa(\Omega)\neq0\).
Equivalently,
$\displaystyle         C_-(A)=0$
iff
$\displaystyle         \Finc(\Omega)+A\FKH(\Omega)\in\Ran\LTD(\Omega)$.
The same relation can be expressed through the finite edge concomitant
\[
        \Bedge:\Eedge\times\ker\LTD(\Omega)^\ast\to\mathbb C,\qquad
        \Bedge(C_-(A)r^{-1/2}\Vm,\Psi^\ast)=C_-(A)\kappa(\Omega),
\]
with
$\displaystyle \kappa(\Omega)=\Bedge(r^{-1/2}\Vm,\Psi^\ast(\Omega))\neq0$.
Thus the singular edge cancellation and the Fredholm projection vanish for the
same value of \(A\).
The transform formulation gives an independent representation of that same
value.  In the flat-plate case, the Kutta-normalized Wiener--Hopf solution is
assumed meromorphic near \(\alpha_{KH}\), i.e., 
$\displaystyle         \Mcal(\alpha;\Omega,\Gcal)
        =
        \frac{\Ncal(\alpha;\Omega,\Gcal)}
             {\Dcal(\alpha;\Omega,\Gcal)} $.
With the wake normalization \(\ell_{KH}\) of \eqref{eq:ellKH-def}, the
coefficient of the downstream wake mode is the residue at \(\alpha_{KH}\)
multiplied by the explicit contour-orientation factor \(i\) of
\eqref{eq:contour-deformation}.  For finite-angle wedges with self-similar
sheet data, the Fourier representation is replaced by a Mellin representation
\[
        \Mcal_{\rm wedge}(s;\Omega,\Gcal)
        =
        \frac{\Ncal_{\rm wedge}(s;\Omega,\Gcal)}
             {\Dcal_{\rm wedge}(s;\Omega,\Gcal)},
        \qquad
        \Arec^{\rm wedge}
        =
        \Res_{s=s_{KH}}\Mcal_{\rm wedge}(s;\Omega,\Gcal),
\]
see Assumption~\ref{ass:wedge-selfsimilar}.
We now state the main result (Figure~\ref{fig:mechanism}).
\begin{theorem}[Fredholm--residue identity for the unsteady Kutta amplitude]
\label{thm:intro-main}
Assume the following hypotheses:
\begin{enumerate}[label=(H\arabic*)]
\item $\ker\Aout^{\rm hom}=\spn\{\Phi_{KH}^{\rm out}\}$, $\Dcal(\alpha_{KH};\omega_{\rm phys},\Gcal)=0$, 
$\Dcal_\alpha(\alpha_{KH};\omega_{\rm phys},\Gcal)\neq0$,
$\Im\alpha_{KH}<0$;
\item $\nabla\phi_A^{\rm out}=C_-(A)r^{-1/2}\Vm+\mathcal O(1)$,
$C_-(A)=C_-^{(0)}+AC_-^{(KH)}$, $C_-^{(KH)}\neq0$;\\
\item $\LTD(\Omega):\Xcal_\sigma\to\Hcal_\sigma$ is Fredholm of index $0$, $\ker\LTD(\Omega)^\ast=\spn\{\Psi^\ast(\Omega)\}$, and for every $A\in\mathbb C$ the matching problem is solvable in 
the augmented class $\Dmatch=\Eedge\oplus\Hcal_\sigma$;
\item $\kappa(\Omega)=\Bedge(r^{-1/2}\Vm,\Psi^\ast(\Omega))\neq0$
(with {\rm(H2)}--{\rm(H3)} this implies
$\langle\FKH,\Psi^\ast\rangle\neq0$ and
\eqref{eq:intro-KF-consistency} with $\chi=-\kappa$);\\
\item $\Mcal(\alpha;\Omega,\Gcal)=\Ncal(\alpha;\Omega,\Gcal)/
\Dcal(\alpha;\Omega,\Gcal)$ is the Kutta-normalized meromorphic transform response near $\alpha_{KH}$.
\end{enumerate}
Then bounded viscous--inviscid matching
 selects a unique amplitude
\[
        \Arec(\Omega,\Gcal)
        =
        -\frac{C_-^{(0)}}{C_-^{(KH)}}
        =
        -\frac{
        \langle \Finc(\Omega),\Psi^\ast(\Omega)\rangle_{\Hcal_\sigma}}
        {
        \langle \FKH(\Omega),\Psi^\ast(\Omega)\rangle_{\Hcal_\sigma}}
        =
        i\Res_{\alpha=\alpha_{KH}}
        \Mcal(\alpha;\Omega,\Gcal),
\]
the last equality holding for the wake normalization
\(\ell_{KH}(\Phi_{KH}^{\rm out})=1\) of \eqref{eq:ellKH-def}.
If \(\alpha_{KH}\) is simple, then
$\displaystyle         \Arec(\Omega,\Gcal)
        =
        \frac{i\,\Ncal(\alpha_{KH};\Omega,\Gcal)}
             {\partial_\alpha\Dcal(\alpha_{KH};\Omega,\Gcal)}$.
Equivalently,
$\displaystyle         C_-(A)=0$
iff
$\displaystyle         \Pi_{\rm sing}F_{\rm match}(A)=0$
iff
$\displaystyle         \langle\Finc+A\FKH,\Psi^\ast\rangle_{\Hcal_\sigma}=0$
iff
$\displaystyle         \Finc+A\FKH\in\Ran\LTD(\Omega)$.
For a finite-angle wedge satisfying the analogous Mellin hypotheses, including
the self-similarity Assumption~\ref{ass:wedge-selfsimilar},
$\displaystyle         \Arec^{\rm wedge}(\Omega,\Gcal)
        =
        -\frac{
        \langle \Finc^{\rm wedge},\Psi_{\rm wedge}^\ast\rangle_{\Hcal_\sigma}}
        {
        \langle \FKH^{\rm wedge},\Psi_{\rm wedge}^\ast\rangle_{\Hcal_\sigma}}
        =
        \Res_{s=s_{KH}}
        \Mcal_{\rm wedge}(s;\Omega,\Gcal)$.
\end{theorem}
\begin{proof}[Proof at the structural level]
By (H1), every outgoing forced outer solution is
\(\Phi_A^{\rm out}=\Phi_0^{\rm out}+A\Phi_{KH}^{\rm out}\).  By (H2), its
singular edge trace is \(C_-(A)r^{-1/2}\Vm\).  Since bounded lower-deck
matching has data in \(\Hcal_\sigma\) and the singular line datum is excluded
from \(\Hcal_\sigma\) by Lemma~\ref{lem:trace-exclusion}, the
\(\Eedge\)-component of \(F_{\rm match}(A)\) must vanish; hence \(C_-(A)=0\),
which gives \(A=-C_-^{(0)}/C_-^{(KH)}\).  By (H3), the regular problem is
solvable exactly when the Fredholm projection against \(\Psi^\ast\) vanishes.
By (H3)--(H4) and Proposition~\ref{prop:KF-noncirc}, that Fredholm projection equals
\(-\kappa(\Omega)C_-(A)\), so the Fredholm-selected and Kutta-selected
amplitudes coincide, giving the adjoint quotient.  Finally, by (H5) and
uniqueness of the Kutta-normalized outer field, the transform solution has the
same amplitude; the causal inverse-transform deformation of
\Cref{subsec:flat-residue} identifies this coefficient, in the normalization
\(\ell_{KH}(\Phi_{KH}^{\rm out})=1\), with \(i\) times the residue at
\(\alpha_{KH}\).  The simple-pole formula is the Laurent coefficient of
\(\Ncal/\Dcal\).
\end{proof}
The theorem is conditional in the following explicit places:
\[
        \text{simple outgoing wake pole},\qquad
        \text{edge indicial root }\lambda=1/2,
\]
\[
        \text{Fredholm lower-deck realization with augmented solvability},
\]
\[
        \text{nonzero edge concomitant},\qquad
        \text{meromorphic Kutta-normalized transform}.
\]
Figure~\ref{fig:mechanism-flowchart} shows the flowchart of our model. 
\Cref{sec:outer,sec:inner} develop each object and prove the algebraic
implications; \Cref{sec:model} verifies the inner hypotheses {\rm(H3)}--{\rm(H4)}
in closed form for the linear-shear model, where the adjoint state is an
Airy-derivative field, the Wiener--Hopf kernel and wake pole are explicit, and
\(\kappa(\Omega)\neq0\) off a discrete resonance set; \Cref{sec:transform}
gives the transform representation and pole-residue formulae;
\Cref{sec:conclusion} discusses scope and limitations;
Appendix~\ref{app:formal-adjoint} derives the formal adjoint and edge concomitant; and
Appendix~\ref{app:wedge} outlines the Mellin analogue for finite-angle wedges.
\begin{figure}[t]
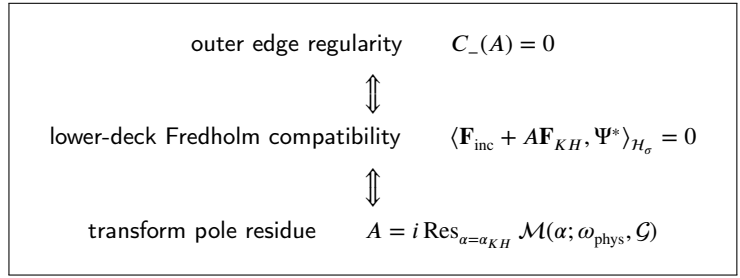

\centering
\setlength{\fboxsep}{10pt}%
\fbox{$\begin{array}{c}
\text{outer edge regularity}\qquad C_-(A)=0\\[2mm]
\Big\Updownarrow\\[2mm]
\text{lower-deck Fredholm compatibility}\qquad
\langle\Finc+A\FKH,\Psi^\ast\rangle_{\Hcal_\sigma}=0\\[2mm]
\Big\Updownarrow\\[2mm]
\text{transform pole residue}\qquad
A=i\Res_{\alpha=\alpha_{KH}}\Mcal(\alpha;\omega_{\rm phys},\Gcal)
\end{array}$}
\caption{The selection mechanism of \Cref{thm:intro-main}: the unsteady Kutta
amplitude is the same scalar in three representations.  The equivalences are
\Cref{thm:conditional-viscous-Kutta} and \Cref{thm:flat-pole-residue}; the
nondegeneracies making them equivalences are {\rm(H2)}--{\rm(H4)}.}
\label{fig:mechanism}
\end{figure}
\clearpage
% =====================================================================
\section{Outer acoustic--wake problem and the edge obstruction}
\label{sec:outer}
Let
\[
        \Gamma_p=(-\infty,0)\times\{0\},\qquad
        \Gamma_w=(0,\infty)\times\{0\},\qquad
        \Pi=\mathbb R^2\setminus(\Gamma_p\cup\Gamma_w).
\]
The edge is \(O=(0,0)\).  The incident acoustic field is denoted by
\(\phi^{\rm inc}\); the total outer potential is
\[
        \phi=\phi^{\rm inc}+\phi^{\rm sc},\qquad
        \phi^\pm(x)=\lim_{y\to0^\pm}\phi(x,y).
\]
We use the convention \(e^{-i\omega_{\rm phys}t}\), and write
\[
        M=\frac Uc\in(0,1),\qquad
        \beta=(1-M^2)^{1/2},\qquad
        k_0=\frac{\omega_{\rm phys}}{c},\qquad
        k=\frac{\omega_{\rm phys}L}{U}.
\]
As stated in \Cref{sec:intro-main}, the outer coordinates are measured on the
hydrodynamic length \(\ell_\omega=U/\omega_{\rm phys}\), so that the outer
problem is uniformly \(\mathcal O(1)\) in \(Re\) throughout the distinguished regime
\eqref{eq:distinguished-LD-frequency}. 
The physical configuration of the problem is illustrated in Figure~\ref{fig:flow_schematic}, which consists of a semi-infinite rigid plate ($\Gamma_p$) and a downstream wake sheet ($\Gamma_w$). 

\begin{figure}[htbp]
    \centering
    \includegraphics[width=0.8\textwidth]{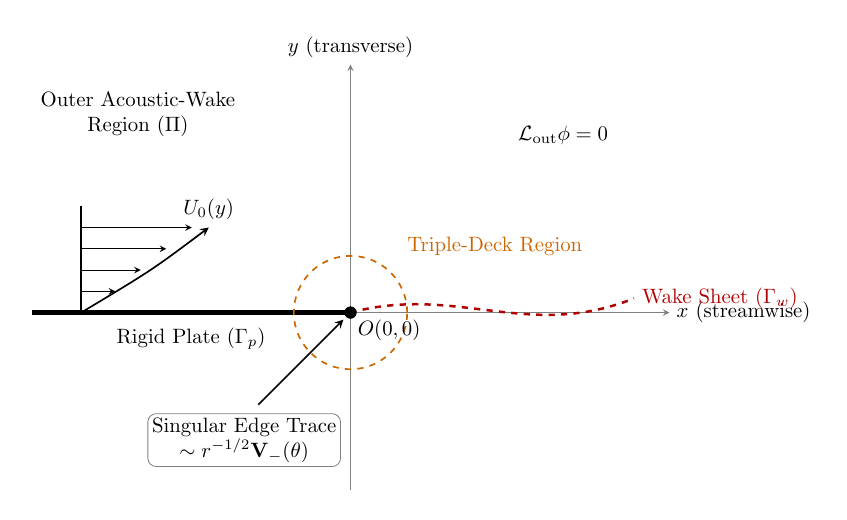}
    \caption{Schematic of the trailing-edge flow configuration. The physical domain includes the semi-infinite rigid plate $\Gamma_p$, the downstream wake sheet $\Gamma_w$, and the incoming boundary-layer velocity profile $U_0(y)$. The local triple-deck region and the outer acoustic-wake region around the trailing edge $O(0,0)$ are also highlighted.}
    \label{fig:flow_schematic}
\end{figure}
\clearpage
% ---------------------------------------------------------------------
\subsection{Convected Helmholtz field and abstract sheet operator}
\label{subsec:outer-CH-sheet}
In \(\Pi\) the scattered potential satisfies
\begin{equation}
\label{eq:CH}
        \Lout\phi^{\rm sc}=0,\qquad
        \Lout:=\beta^2\partial_x^2+\partial_y^2
        +2iMk_0\partial_x+k_0^2 .
\end{equation}
Equivalently, for the total potential,
\[
        \Lout\phi=0,\qquad
        \phi-\phi^{\rm inc}\ \text{outgoing}.
\]
On the plate with $x<0$,
$\displaystyle \gamma_p^\pm(\partial_y\phi)=0$,
where \(\gamma_p^\pm f=f(x,0^\pm)\).  On the wake we impose a linear
inviscid-sheet transmission law
\begin{equation}
\label{eq:abstract-sheet}
        \Tsh(\omega_{\rm phys},\Gcal)
        \begin{pmatrix}
        \gamma_w^+\phi\\[1mm]
        \gamma_w^-\phi\\[1mm]
        \gamma_w^+\partial_y\phi\\[1mm]
        \gamma_w^-\partial_y\phi\\[1mm]
        \eta
        \end{pmatrix}
        =
        0,\qquad x>0,
\end{equation}
where \(\gamma_w^\pm f=f(x,0^\pm)\), \(\eta(x)e^{-i\omega_{\rm phys}t}\) is
the sheet displacement, and \(\Gcal\) denotes the edge geometry, acoustic
incidence, and base sheet data.  Thus
\[
        \Gcal=(M,\omega_{\rm phys},\hbox{edge angle},\hbox{sheet strength},
        \hbox{incidence data},\ldots).
\]
The neutral equal-speed relations
\[
        \partial_y\phi^\pm=(-i\omega_{\rm phys}+U\partial_x)\eta,\qquad
        (-i\omega_{\rm phys}+U\partial_x)[\phi]=0,\qquad
        [\phi]:=\phi^+-\phi^-,
\]
are regarded only as a limiting convected-sheet model.  The Kelvin--Helmholtz
branch used below is attached to the full operator \(\Tsh\), not to this
degenerate equal-\(U\) limit~\cite{orszag1970instability,crighton1974radiation}.
It is useful to collect the boundary data into
        $\displaystyle \mathbf g_w(\phi,\eta)
        :=
        \big(\gamma_w^+\phi,\gamma_w^-\phi,
        \gamma_w^+\partial_y\phi,\gamma_w^-\partial_y\phi,\eta\big)^{\mathsf T}$.
The outer problem is therefore
\begin{equation}
\label{eq:outer-operator-problem}
        \Aout(\omega_{\rm phys},\Gcal)
        \begin{pmatrix}\phi\\ \eta\end{pmatrix}
        =
        \begin{pmatrix}
        \Lout\phi\\
        \gamma_p^\pm\partial_y\phi\\
        \Tsh(\omega_{\rm phys},\Gcal)\mathbf g_w(\phi,\eta)
        \end{pmatrix}
        =
        \Fout^{\rm inc},
\end{equation}
with outgoing radiation for the acoustic part and downstream causality for the
hydrodynamic sheet modes.  Here \(\Fout^{\rm inc}\) is the boundary forcing
obtained after subtracting \(\phi^{\rm inc}\).
For later use we introduce the local trace spaces
\[
        \mathfrak H_{\rm out}
        :=
        H^1_{\rm loc}(\Pi)\times H^{1/2}_{\rm loc}(\Gamma_w),
        \qquad
        \mathfrak Y_{\rm out}
        :=
        H^{-1}_{\rm loc}(\Pi)\times
        H^{-1/2}_{\rm loc}(\Gamma_p)\times
        \mathfrak Z_w ,
\]
and regard
$\displaystyle         \Aout(\omega_{\rm phys},\Gcal):
        \mathfrak D(\Aout)\subset\mathfrak H_{\rm out}
        \longrightarrow
        \mathfrak Y_{\rm out}$
as the outgoing outer acoustic--wake operator.  The precise Banach realization
is not needed below; only the one-dimensional kernel and the
local edge trace are used.
% ---------------------------------------------------------------------
\subsection{Assumed simple downstream wake pole}
\label{subsec:simple-KH-pole}
For a normal mode
        $\displaystyle (\phi,\eta)(x,y)=e^{i\alpha x}\,
        \big(\varphi^+(y),\varphi^-(y),\eta_0\big)$,
\eqref{eq:CH} gives
$\displaystyle         (\varphi^\pm)''-\mu(\alpha)^2\varphi^\pm=0$ with
$\displaystyle         \mu(\alpha)^2=\beta^2\alpha^2+2Mk_0\alpha-k_0^2 $.
The branch is fixed by the outgoing/decaying condition
\[
        \Re\mu(\alpha)>0
        \quad\text{on the physical inversion contour}.
\]
Solving \(\mu(\alpha)^2=0\) with \(\beta^2=(1-M)(1+M)\) gives
\(\alpha=k_0(-M\pm1)/\beta^2\), i.e. the acoustic branch points
\begin{equation}
\label{eq:branch-points}
        \alpha_+=\frac{k_0}{1+M},\qquad
        \alpha_-=-\frac{k_0}{1-M},
        \qquad\text{i.e. } \alpha_\pm=\pm\frac{k_0}{1\pm M}.
\end{equation}
These are the downstream-propagating wavenumber \(\omega_{\rm phys}/(c+U)\) and
the upstream wavenumber \(-\omega_{\rm phys}/(c-U)\).
The sheet law \eqref{eq:abstract-sheet} reduces the homogeneous normal-mode
problem to a finite-dimensional algebraic system
\[
        \Bcal(\alpha;\omega_{\rm phys},\Gcal)\mathbf a=0,
        \qquad
        \mathbf a=(a_+,a_-,b_+,b_-,\eta_0)^{\mathsf T}.
\]
Its dispersion determinant is
        $\displaystyle \Dcal(\alpha;\omega_{\rm phys},\Gcal)
        :=
        \det \Bcal(\alpha;\omega_{\rm phys},\Gcal)$,
after removal of nonphysical normalization factors.
\begin{assumption}[Simple downstream wake pole]
\label{ass:simple-KH}
There exists \(\alpha_{KH}\in\mathbb C\) such that
\begin{equation}
\label{eq:KH-pole-ass}
        \Dcal(\alpha_{KH};\omega_{\rm phys},\Gcal)=0,\qquad
        \partial_\alpha\Dcal(\alpha_{KH};\omega_{\rm phys},\Gcal)\neq0,
        \qquad
        \Im\alpha_{KH}<0 ,
\end{equation}
and
        $\displaystyle \ker\Bcal(\alpha_{KH};\omega_{\rm phys},\Gcal)
        =
        \spn\{\mathbf a_{KH}\}$.
The associated outgoing homogeneous field is denoted
\[
        \Phi_{KH}^{\rm out}
        :=
        (\phi_{KH}^{\rm out},\eta_{KH}),\qquad
        \phi_{KH}^{\rm out}(x,y)
        \sim e^{i\alpha_{KH}x}\varphi_{KH}(y),
        \quad x\to+\infty .
\]
\end{assumption}
With the convention \(e^{i\alpha x-i\omega_{\rm phys}t}\), the inequality
\(\Im\alpha_{KH}<0\) corresponds to downstream spatial growth---a convectively
unstable wake mode in the Briggs--Bers sense~\cite{briggs1964electron,bers1983space,
monkewitz1990local}; the causal contour argument that justifies collecting it
downstream is given in \Cref{subsec:flat-residue}.  The normalization is fixed
once and for all by the explicit wake functional
\begin{equation}
\label{eq:ellKH-def}
        \ell_{KH}(\Phi)
        :=
        \lim_{x\to+\infty}e^{-i\alpha_{KH}x}\,\eta(x),
        \qquad
        \ell_{KH}(\Phi_{KH}^{\rm out})=1 ,
\end{equation}
i.e.\ the normalized wake mode has unit displacement amplitude,
\(\eta_{KH}(x)=e^{i\alpha_{KH}x}\).
We also record the abstract homogeneous assumption used below:
\begin{equation}
\label{eq:outer-kernel-ass}
        \ker \Aout^{\rm hom}(\omega_{\rm phys},\Gcal)
        =
        \spn\{\Phi_{KH}^{\rm out}\}.
\end{equation}
The analysis does not require an explicit formula for \(\Dcal\).  It requires
only \eqref{eq:KH-pole-ass} and \eqref{eq:outer-kernel-ass}.
% ---------------------------------------------------------------------
\subsection{One-dimensional kernel}
\label{subsec:outer-ambiguity}
Let
\[
        \Phi=(\phi,\eta),\qquad
        \Phi_0^{\rm out}=(\phi_0^{\rm out},\eta_0)
\]
be one outgoing forced solution of \eqref{eq:outer-operator-problem}.  By
\eqref{eq:outer-kernel-ass},
        $\displaystyle \Phi_A^{\rm out}
        =
        \Phi_0^{\rm out}+A\Phi_{KH}^{\rm out}$ for
        $A\in\mathbb C$,
is the complete outgoing affine family with the same incident acoustic forcing.
\begin{proposition}[One-dimensional kernel]
\label{prop:outer-ambiguity}
Assume \eqref{eq:outer-kernel-ass}.  If
\(\Aout\Phi_0^{\rm out}=\Fout^{\rm inc}\), then
        $\displaystyle \Aout\Phi_A^{\rm out}=\Fout^{\rm inc}$ for any $A\in\mathbb C$.
Conversely, if \(\Aout\widetilde\Phi^{\rm out}=\Fout^{\rm inc}\) and
\(\widetilde\Phi^{\rm out}\) satisfies the same outgoing convention, then
        $\displaystyle \widetilde\Phi^{\rm out}
        =
        \Phi_0^{\rm out}+A\Phi_{KH}^{\rm out}$
for a unique \(A\in\mathbb C\).
\end{proposition}
\begin{proof}
Linearity gives
\(\Aout(\Phi_0^{\rm out}+A\Phi_{KH}^{\rm out})
=\Fout^{\rm inc}+A\Aout\Phi_{KH}^{\rm out}=\Fout^{\rm inc}\).  If
\(\widetilde\Phi^{\rm out}\) is another outgoing solution, then
\(\widetilde\Phi^{\rm out}-\Phi_0^{\rm out}\in\ker\Aout^{\rm hom}\), hence
\(\widetilde\Phi^{\rm out}-\Phi_0^{\rm out}=A\Phi_{KH}^{\rm out}\).  Uniqueness
of \(A\) follows from \(\Phi_{KH}^{\rm out}\neq0\).
\end{proof}
Thus the outer inviscid problem fixes the acoustic field only modulo
\(\Phi_{KH}^{\rm out}\).  The scalar
        $A=\ell_{KH}(\Phi_A^{\rm out}-\Phi_0^{\rm out})$
is the outer receptivity amplitude.
% ---------------------------------------------------------------------
\subsection{Local edge expansion and the coefficient \texorpdfstring{\(C_-(A)\)}{C-(A)}}
\label{subsec:edge-expansion}
Let \((r,\theta)\), \(0<r<r_0\), \(-\pi<\theta<\pi\), denote polar
coordinates after the local stretching
\[
        X_o=\frac{x}{\beta},\qquad Y_o=y ;
\]
the coefficients in the unstretched frame differ by explicit powers of
\(\beta\), immaterial to the affine structure in \(A\) used below.  
Near \(O\),
        $\displaystyle \Lout
        =
        \beta^2\partial_x^2+\partial_y^2+\text{lower-order terms}$,
so after the stretching the principal symbol is \(|\xi|^2+|\zeta|^2\).  Hence
the indicial operator is the slit-plane Laplacian with the principal edge
transmission constraints; the Kondrat'ev theory of corner
asymptotics~\cite{kondrat1967boundary} applies.
We assume the first nonconstant indicial root is the Neumann slit-plane root.
\begin{assumption}[Edge indicial structure]
\label{ass:edge-indicial}
The principal edge pencil \(\mathfrak P(\lambda;\Gcal)\) has
        $\displaystyle \lambda_0=0$ and
        $\displaystyle \lambda_1=\frac12$,
where \(\lambda_1\) is simple modulo the constant mode and carries no
logarithmic terms.  The corresponding angular function may be chosen as
        $\displaystyle \Psi_-(\theta)=\sin\frac{\theta}{2}$
after the local elliptic stretching.  The lower-order convective/acoustic terms
and the sheet trace equations do not shift \(\lambda_1\); they only determine
higher coefficients and linear relations among the edge amplitudes.
\end{assumption}
For the principal slit problem,
\[
        \partial_{\theta}\Psi(\pm\pi)=0,\qquad
        \Psi''+\lambda^2\Psi=0,
\]
so the indicial roots are \(\lambda_n=n/2\), with angular functions
\[
        \Psi_n(\theta)=
        \begin{cases}
        \sin\dfrac{n\theta}{2}, & n\ \text{odd},\\[2mm]
        \cos\dfrac{n\theta}{2}, & n\ \text{even},
        \end{cases}
\]
the Neumann conditions at \(\theta=\pm\pi\) selecting alternating parities.
In particular the first nonconstant mode is \(\Psi_1=\Psi_-=\sin(\theta/2)\).
The singular velocity profile associated with \(\Psi_-\) is
\begin{equation}
\label{eq:Vminus-def}
        \Vm(\theta)
        =
        \frac12\Psi_-(\theta)e_r+\Psi_-'(\theta)e_\theta,
        \qquad
        e_r=(\cos\theta,\sin\theta),\quad
        e_\theta=(-\sin\theta,\cos\theta).
\end{equation}
Equivalently,
\begin{equation}
\label{eq:Vminus-components}
        \Vm\cdot e_x=-\frac12\sin\frac{\theta}{2},\qquad
        \Vm\cdot e_y= \frac12\cos\frac{\theta}{2}.
\end{equation}
Thus
\[
        \Vm\cdot e_x|_{\theta=\pi}=-\frac12,\qquad
        \Vm\cdot e_x|_{\theta=-\pi}=+\frac12,\qquad
        \Vm\cdot e_y|_{\theta=0}=\frac12 .
\]
Define the edge coefficient \(C_-(\Phi)\) by the asymptotic projection
\begin{equation}
\label{eq:Cminus-projection}
        C_-(\Phi)
        :=
        \lim_{\rho\downarrow0}
        \frac{
        \displaystyle
        \int_{-\pi}^{\pi}
        \big(\phi(\rho,\theta)-\bar\phi_\rho\big)\Psi_-(\theta)\,d\theta}
        {\displaystyle
        \rho^{1/2}\int_{-\pi}^{\pi}\Psi_-(\theta)^2\,d\theta},
        \qquad
        \bar\phi_\rho=\frac1{2\pi}\int_{-\pi}^{\pi}\phi(\rho,\theta)\,d\theta .
\end{equation}
Equivalently, \(C_-\) is the coefficient of the \(r^{1/2}\Psi_-\) term in the
Kondrat'ev expansion.
\begin{proposition}[Edge expansion]
\label{prop:edge-expansion}
Assume Assumption~\ref{ass:edge-indicial}.  For each \(A\in\mathbb C\),
\[
        \phi_A^{\rm out}(r,\theta)
        =
        \phi_e(A)
        +
        C_-(A)r^{1/2}\Psi_-(\theta)
        +
        r\Psi_0(A,\theta)
        +
        \mathcal O(r^{3/2})
\]
in \(H^1\)-conormal form as \(r\downarrow0\).  Hence
\begin{equation}
\label{eq:grad-edge-expansion}
        \nabla\phi_A^{\rm out}(r,\theta)
        =
        C_-(A)r^{-1/2}\Vm(\theta)
        +
        \Vz(A,\theta)
        +
        \mathcal O(r^{1/2}).
\end{equation}
Moreover
\begin{equation}
\label{eq:Cminus-affine}
        C_-(A)=C_-^{(0)}+A\,C_-^{(KH)},
        \qquad
        C_-^{(0)}:=C_-(\Phi_0^{\rm out}),\quad
        C_-^{(KH)}:=C_-(\Phi_{KH}^{\rm out}).
\end{equation}
\end{proposition}
\begin{proof}
The local elliptic pencil gives the conormal expansion
\[
        \phi_A^{\rm out}
        =
        \sum_{\lambda\in\Lambda,\ \Re\lambda<3/2}
        r^\lambda\Psi_\lambda(\theta)c_\lambda(A)
        +O(r^{3/2}).
\]
By Assumption~\ref{ass:edge-indicial}, the only terms below \(3/2\) relevant to the
singular velocity are \(\lambda=0\), \(\lambda=1/2\), and \(\lambda=1\), with
no logarithms.  Thus the displayed expansion follows.  Since
$\displaystyle         \nabla(r^{1/2}\Psi_-)
        =
        r^{-1/2}\left(\frac12\Psi_-e_r+\Psi_-'e_\theta\right)$,
\eqref{eq:grad-edge-expansion} follows from \eqref{eq:Vminus-def}.  Finally,
\(\Phi_A^{\rm out}=\Phi_0^{\rm out}+A\Phi_{KH}^{\rm out}\) and the projection
\eqref{eq:Cminus-projection} is linear; hence \eqref{eq:Cminus-affine}.
\end{proof}
The singular pressure is obtained from the linearized Bernoulli relation
\[
        p[\phi]=-\rho_0(-i\omega_{\rm phys}+U\partial_x)\phi .
\]
Since \((-i\omega_{\rm phys})\,r^{1/2}\Psi_-=\mathcal O(r^{1/2})\) and
\(\partial_x(r^{1/2}\Psi_-)=r^{-1/2}\Vm\cdot e_x\),
\begin{equation}
\label{eq:pressure-edge-sing}
        p[\phi_A^{\rm out}]
        =
        -\rho_0U\,C_-(A)r^{-1/2}\Vm(\theta)\cdot e_x
        +\mathcal O(1).
\end{equation}
Thus the same scalar \(C_-(A)\) controls the inverse-square-root velocity
singularity and the leading pressure singularity.
We introduce the singular edge trace space
        $\displaystyle \Eedge
        :=
        \spn\{r^{-1/2}\Vm(\theta)\}$,
and the regular edge data space \(\Redge\) by
\[
        \nabla\phi_A^{\rm out}
        =
        C_-(A)\,r^{-1/2}\Vm+\mathcal R_A,\qquad
        \mathcal R_A\in\Redge .
\]
Hence
        $\displaystyle \Tssing\nabla\phi_A^{\rm out}
        =
        C_-(A)\,r^{-1/2}\Vm
        \in\Eedge $.
On the lower-deck scale \(r=\varepsilon^3 L\,(X^2+\varepsilon^4Y^2)^{1/2}\),
the trace of \(\Eedge\) on the matching components is an
\(|X|^{-1/2}\)-profile line datum.
% ---------------------------------------------------------------------
\subsection{Kutta regularity as \texorpdfstring{\(C_-(A)=0\)}{C-(A)=0}}
\label{subsec:kutta-regularity}
The outer regularity form of the Kutta condition is the annihilation of the
singular edge trace:
\begin{equation}
\label{eq:kutta-Cminus}
        \Tssing\nabla\phi_A^{\rm out}=0
        \quad\Longleftrightarrow\quad
        C_-(A)=0 .
\end{equation}
By \eqref{eq:Cminus-affine},
\begin{equation}
\label{eq:A-kutta}
        A_{\rm Kutta}^{\rm out}
        =
        -\,\frac{C_-^{(0)}}{C_-^{(KH)}},
        \qquad
        C_-^{(KH)}\neq0 .
\end{equation}
\begin{lemma}[Uniqueness of the Kutta-normalized outer field]
\label{lem:kutta-unique}
Assume \eqref{eq:outer-kernel-ass}, Assumption~\ref{ass:edge-indicial}, and
\(C_-^{(KH)}\neq0\).  Then there exists a unique \(A\in\mathbb C\) such that
\(\Phi_A^{\rm out}\) satisfies \eqref{eq:kutta-Cminus}.  It is given by
\eqref{eq:A-kutta}.  If \(\widetilde\Phi^{\rm out}\) is any outgoing solution
with the same incident forcing and \(C_-(\widetilde\Phi^{\rm out})=0\), then
\(\widetilde\Phi^{\rm out}=\Phi_{A_{\rm Kutta}^{\rm out}}^{\rm out}\).
\end{lemma}
\begin{proof}
Every outgoing solution is \(\Phi_A^{\rm out}\) by
Proposition~\ref{prop:outer-ambiguity}.  The Kutta condition is
\(C_-^{(0)}+A C_-^{(KH)}=0\).  Since \(C_-^{(KH)}\neq0\), the solution is unique
and equals \eqref{eq:A-kutta}.
\end{proof}
Equivalently,
$       \displaystyle  \mathfrak K:
        \Phi_A^{\rm out}\mapsto C_-(A)$
is a nontrivial linear functional on the one-dimensional kernel.  The
outer Kutta quotient is therefore
$\displaystyle         A_{\rm Kutta}^{\rm out}
        =
        -\frac{\mathfrak K(\Phi_0^{\rm out})}
              {\mathfrak K(\Phi_{KH}^{\rm out})}$.
At this stage \eqref{eq:kutta-Cminus} is only an outer regularity condition.
The viscous lower-deck analysis will identify it with a Fredholm compatibility
condition:
\[
        C_-(A)=0
        \quad\Longleftrightarrow\quad
        \Finc+A\FKH\in\Ran\LTD(\Omega)
        \quad\Longleftrightarrow\quad
        \langle \Finc+A\FKH,\Psi^\ast\rangle_{\Hcal}=0 .
\]
Thus the scalar obstruction passed from the outer problem to the inner problem
is precisely
$\displaystyle         C_-(A)=C_-^{(0)}+A C_-^{(KH)}$ .
% =====================================================================
\section{Unsteady lower deck and Fredholm Kutta selection}
\label{sec:inner}
The outer field of \Cref{sec:outer} has
\[
        \nabla\phi_A^{\rm out}
        =
        C_-(A)r^{-1/2}\Vm+\mathcal R_A,\qquad
        C_-(A)=C_-^{(0)}+A C_-^{(KH)} .
\]
We now derive the scalar condition selecting \(A\) from the viscous
trailing-edge region.  The viscous structure is the unsteady triple
deck~\cite{stewartson1969flow,messiter1970boundary,stewartson1974multistructured,sychev1998asymptotic}.
% ---------------------------------------------------------------------
\subsection{Triple-deck scaling and distinguished frequency}
\label{subsec:TD-scaling}
Let
\[
        Re=\frac{UL}{\nu},\qquad
        \varepsilon=Re^{-1/8},\qquad
        x=\varepsilon^3L\,X,\qquad
        y=\varepsilon^5L\,Y .
\]
The main-deck thickness is \(\mathcal O(\varepsilon^4L)\); hence the lower-deck shear
speed at \(y=\mathcal O(\varepsilon^5L)\) is
\[
        U_{LD}=\mathcal O(\varepsilon U),\qquad
        \ell_{LD}=\mathcal O(\varepsilon^3L),\qquad
        t_{LD}=\frac{\ell_{LD}}{U_{LD}}
        =
        \frac{\varepsilon^2L}{U}.
\]
Thus
\[
        T=\frac{t}{t_{LD}}=\frac{Ut}{\varepsilon^2L},\qquad
        e^{-i\omega_{\rm phys}t}=e^{-i\Omega T},
        \qquad
        \Omega=\omega_{\rm phys}t_{LD}
        =
        \frac{\omega_{\rm phys}\varepsilon^2L}{U}.
\]
With \(k=\omega_{\rm phys}L/U\),
$\displaystyle         \Omega=\varepsilon^2 k=Re^{-1/4}k$ .
The distinguished unsteady lower-deck regime is
\begin{equation}
\label{eq:distinguished-LD-frequency}
        \Omega=\mathcal O(1),\qquad
        k=\mathcal O(\varepsilon^{-2})=\mathcal O(Re^{1/4}).
\end{equation}
Use the lower-deck scales
\[
        u_{\rm phys}=\varepsilon U\,U(X,Y,T),\qquad
        v_{\rm phys}=\varepsilon^3 U\,V(X,Y,T),\qquad
        p_{\rm phys}=\rho_0\varepsilon^2U^2P(X,T).
\]
Then the leading lower-deck equations are
$\displaystyle         U_X+V_Y=0$ and
$U_T+UU_X+VU_Y=-P_X+U_{YY}$.
% ---------------------------------------------------------------------
\subsection{Symmetry components and the two-sided wake}
\label{subsec:two-sided}
Downstream of the edge the deck occupies \(Y\in\mathbb R\): the trailing-edge
lower deck consists of two wall layers for \(X<0\), \(\pm Y>0\), merging into a
two-sided wake layer for \(X>0\), with a smooth symmetric steady base state
\[
        U_0(X,-Y)=U_0(X,Y),\qquad
        V_0(X,-Y)=-V_0(X,Y),\qquad
        U_{0Y}(X,0)=0\ \ (X>0),
\]
cf.~\cite{stewartson1969flow,messiter1970boundary,jobe1974numerical}.  Unsteady
perturbations carry, in addition to \((u,v,p,a^\pm)\), a wake-centerline
displacement \(h(X)e^{-i\Omega T}\), with linearized centerline conditions for
\(X>0\)
\begin{equation}
\label{eq:centerline-conditions}
        [u]=0,\qquad [u_Y]=0,\qquad
        v(X,0^\pm)=(-i\Omega+U_c(X)\partial_X)h,\qquad
        U_c(X):=U_0(X,0),
\end{equation}
where \([\,\cdot\,]\) denotes the jump across \(Y=0\) (the base smoothness
\(U_{0Y}(X,0)=0\) removes base-shear jump terms).  Because the base state is
symmetric, the linearized problem decomposes into a symmetric
component (\(u\) even, \(v\) odd in \(Y\), \(h=0\)) and an
antisymmetric component (\(u\) odd, \(v\) even, \(h\neq0\)), each with
its own pressure--displacement interaction map.  On the half-plane
\(Y>0\) these reduce to
\begin{align}
\label{eq:LD-bc}
        &\text{(both components)}\quad u=v=0\quad(X<0,\,Y=0^+);\\
\label{eq:LD-bc-sym}
        &\text{(symmetric)}\quad v=0,\quad u_Y=0\quad(X>0,\,Y=0^+);\\
\label{eq:LD-bc-antisym}
        &\text{(antisymmetric)}\quad u=0,\quad
        v=(-i\Omega+U_c\partial_X)h\quad(X>0,\,Y=0^+).
\end{align}
The Kelvin--Helmholtz wake mode and its matching data are antisymmetric (the
edge angular function \(\sin(\theta/2)\) is odd); incident acoustic data
generically force both components.  All structural statements of this section
(weighted realization, Fredholm hypothesis, adjoint, edge concomitant, and the
selection \Cref{thm:conditional-viscous-Kutta}) are formulated componentwise:
\(\LTD(\Omega)\) denotes the linearized operator of either component on
\(Y>0\), with the corresponding half-plane conditions and interaction law.
For notational ease the displayed formulas below are written for the
symmetric component \eqref{eq:LD-bc-sym}, for which the worked example of
\Cref{sec:model} is carried out in closed form; the antisymmetric component
differs only in the wake-side boundary block and in the explicit wake
impedance, and is discussed in Remark~\ref{rem:model-antisym}.
The upper matching condition is
\[
        U(X,Y,T)=\lambda_0Y+\Delta(X,T)+o(1),
        \qquad Y\to+\infty,\qquad \lambda_0>0 .
\]
For the flat-plate subsonic upper-deck map of the relevant component,
\begin{equation}
\label{eq:LD-interaction}
        P=\Kcal[\Delta],\qquad
        \Kcal=H\partial_X,\qquad
        \widehat{\Kcal a}(\alpha)=|\alpha|\widehat a(\alpha),
\end{equation}
where \(H\) is the Hilbert transform.  For wedge geometries
\(\Kcal\) is replaced by the corresponding wedge pressure--displacement
operator.
Let \((U_0,V_0,P_0,\Delta_0)\) be a steady lower-deck solution:
\[
\begin{gathered}
        U_{0X}+V_{0Y}=0,\qquad
        U_0U_{0X}+V_0U_{0Y}=-P_{0X}+U_{0YY},\\
        U_0\sim\lambda_0Y+\Delta_0(X)\quad(Y\to\infty),\qquad
        P_0=\Kcal[\Delta_0],
\end{gathered}
\]
with \eqref{eq:LD-bc}--\eqref{eq:LD-bc-sym}.  In particular,
\begin{equation}
\label{eq:base-identities}
        U_{0X}+V_{0Y}=0,\qquad
        V_0|_{X>0,Y=0}=0 .
\end{equation}
% ---------------------------------------------------------------------
\subsection{Linearized unsteady lower-deck operator}
\label{subsec:linear-LD-operator}
Set
\[
        U=U_0+u e^{-i\Omega T},\quad
        V=V_0+v e^{-i\Omega T},\quad
        P=P_0+p e^{-i\Omega T},\quad
        \Delta=\Delta_0+a e^{-i\Omega T}.
\]
The linearized system is
\begin{equation}
\label{eq:lin-LD}
\begin{aligned}
        u_X+v_Y&=f_0,\\
        -i\Omega u+U_0u_X+V_0u_Y+U_{0X}u+U_{0Y}v+p_X-u_{YY}&=f_1,\\
        u(X,Y)-a(X)&\to f_\infty(X)\quad(Y\to\infty),\\
        p-\Kcal[a]&=f_K .
\end{aligned}
\end{equation}
The homogeneous boundary conditions are \eqref{eq:LD-bc}--\eqref{eq:LD-bc-sym}.
Write
$\displaystyle         W=(u,v,p,a)^{\mathsf T}$ and 
 $\displaystyle        F=(f_0,f_1,f_\infty,f_K)^{\mathsf T}.$
Then \eqref{eq:lin-LD} with \eqref{eq:LD-bc}--\eqref{eq:LD-bc-sym} defines
$\displaystyle         \LTD(\Omega)W=F $.
The outer-to-inner matching map is denoted
$\displaystyle         \Min:\Phi^{\rm out}\mapsto F_{\rm match}$.
Linearity of matching gives
\begin{equation}
\label{eq:forcing-affine}
        F_{\rm match}(A)
        =
        \Min[\Phi_A^{\rm out}]
        =
        \Finc(\Omega)+A\FKH(\Omega)
        +
        C_-(A)\,\Fsing(\Omega),
\end{equation}
where
\[
        \Finc:=\Min[\Phi_0^{\rm out}]_{\rm reg},\qquad
        \FKH:=\Min[\Phi_{KH}^{\rm out}]_{\rm reg},
\]
the subscript denoting the part of the matching data remaining after the
singular edge trace is split off.  The singular datum is generated by
$\displaystyle         \nabla\phi_A^{\rm out}
        =
        C_-(A)r^{-1/2}\Vm+\mathcal R_A$.
Since
$\displaystyle         r=\varepsilon^3L(X^2+\varepsilon^4Y^2)^{1/2}$,
the singular trace enters the lower deck, at leading order, as the
line datum
\begin{equation}
\label{eq:inner-singular-trace}
        \Fsing
        =
        \big(0,\,0,\,g_\infty^\sharp,\,g_K^\sharp\big),
        \qquad
        g_\infty^\sharp(X)=c_\infty^\pm|X|^{-1/2},\quad
        g_K^\sharp(X)=c_K^\pm|X|^{-1/2}\quad(\pm X>0),
\end{equation}
with the explicit constants
\begin{equation}
\label{eq:singular-trace-constants}
        c_\infty^{+}=0,\qquad
        c_\infty^{-}=\mp\tfrac12\ \ (\theta=\pm\pi),\qquad
        c_K^{+}=0,\qquad
        c_K^{-}=\pm\tfrac12\ \ (\theta=\pm\pi),
\end{equation}
read off from \eqref{eq:Vminus-components} and
\eqref{eq:pressure-edge-sing} (slip and pressure traces upstream; on the wake
side \(\theta=0\) the streamwise trace of \(\Vm\) vanishes and the singular
content is carried by the transverse/displacement trace, which enters the
antisymmetric component analogously).  The normalization of
\(c_\infty^\pm,c_K^\pm\) is fixed once by the matching map and plays no role
beyond the linearity \(\Fsing\mapsto C_-(A)\Fsing\).
The regular lower-deck problem therefore reads
\begin{equation}
\label{eq:LTD-regular-forced}
        \LTD(\Omega)W=\Finc(\Omega)+A\FKH(\Omega),
\end{equation}
provided \(C_-(A)=0\).  The role of the next subsections is to show that this
condition is also forced by Fredholm solvability.
The formal \(L^2\)-adjoint is obtained from the bilinear pairing
$\displaystyle        \iint_{\mathbb R\times\mathbb R_+}
        \{u^\ast R_1+qR_0\}\,dX\,dY$,
where
\[
        R_0=u_X+v_Y,\qquad
        R_1=-i\Omega u+U_0u_X+V_0u_Y+U_{0X}u+U_{0Y}v+p_X-u_{YY}.
\]
Using \eqref{eq:base-identities}, the adjoint bulk equations are
\begin{equation}
\label{eq:formal-adjoint}
        -i\Omega u^\ast-U_0u^\ast_X-V_0u^\ast_Y
        +U_{0X}u^\ast-u^\ast_{YY}-q_X=0,\qquad
        q_Y=U_{0Y}u^\ast ;
\end{equation}
the full derivation, with the line and edge terms, is in
Appendix~\ref{app:formal-adjoint}.
The Lagrange concomitant is
\begin{equation}
\label{eq:LD-concomitant}
        J^X=qu+U_0u^\ast u+u^\ast p,\qquad
        J^Y=qv+V_0u^\ast u-(u^\ast u_Y-u_Y^\ast u).
\end{equation}
The transposed wall/wake conditions are
\[
        u^\ast=0\quad(X<0,Y=0),\qquad
        u_Y^\ast=0\quad(X>0,Y=0),
\]
together with decay at \(Y=\infty\).  Treating \(p=\Kcal[a]\) and
\(u(\cdot,\infty)=a\) with line multipliers gives
\begin{equation}
\label{eq:adjoint-interaction-main}
        b=\bar U_X^\ast,\qquad
        \mu=\Kcal[\bar U_X^\ast]=H[\bar U_{XX}^\ast],
        \qquad
        \bar U^\ast(X):=\int_0^\infty u^\ast(X,Y)\,dY .
\end{equation}
% ---------------------------------------------------------------------
\subsection{Weighted spaces and Fredholm hypothesis}
\label{subsec:weighted-Fredholm}
Let \(\langle X\rangle=(1+X^2)^{1/2}\) and define, for
\(\sigma=(\sigma_-,\vartheta)\) with \(\sigma_->0\) and \(\vartheta>0\),
\begin{equation}
\label{eq:weights-new}
        w_\sigma(X)=
        \begin{cases}
        \langle X\rangle^{\sigma_-},& X<0,\\[1mm]
        e^{-\vartheta X},& X>0,
        \end{cases}
\end{equation}
with \(\vartheta>|\Im\alpha_{KH}^{\rm in}|\), where \(\alpha_{KH}^{\rm in}\)
is the inner wavenumber of the shed wake mode. 
Thus algebraic decay is demanded upstream, while the exponential downstream weight admits the
spatially growing or neutral wake response into the function class; this
choice is what produces the one-dimensional kernel and cokernel below
(Remark~\ref{rem:kernel-interpretation}).  For a scalar field \(g\), set
$\displaystyle         \|g\|_{L^2_\sigma}^2
        =
        \iint_{\mathbb R\times\mathbb R_+}
        |w_\sigma(X)g(X,Y)|^2\,dX\,dY $,
and analogously for line functions.  Define the anisotropic model domain
\begin{equation*}
\begin{aligned}
        \Xcal_\sigma
        :=
        \{W=(u,v,p,a):\;&
        u,u_X,u_{YY},v,v_Y\in L^2_\sigma,\\
        &p,p_X\in L^2_\sigma(\mathbb R),\quad
        a,\Kcal a\in H^1_\sigma(\mathbb R),\\
        &u-a\to0\ (Y\to\infty),\quad
        \eqref{eq:LD-bc}, \eqref{eq:LD-bc-sym}\ \text{hold in trace sense}\},
\end{aligned}
\end{equation*}
a Banach space with its natural norm.  The data space is
\begin{equation}
\label{eq:Hsigma}
        \Hcal_\sigma
        :=
        L^2_\sigma(\Pi)
        \times L^2_\sigma(\Pi)
        \times H^{1/2}_\sigma(\mathbb R)
        \times H^{-1/2}_\sigma(\mathbb R).
\end{equation}
The realization of \(\LTD(\Omega)\) is the bounded operator
$\displaystyle         \LTD(\Omega):\Xcal_\sigma\longrightarrow\Hcal_\sigma$
between these Banach spaces; the adjoint \(\LTD(\Omega)^\ast\) acts on the
dual weight class (in particular, adjoint states decay downstream faster than
\(e^{-\vartheta X}\), localizing the adjoint near the edge and upstream, as in
receptivity theory~\cite{goldstein1989boundary}).
\begin{lemma}[Trace-level exclusion of the singular datum]
\label{lem:trace-exclusion}
\(\Fsing\notin\Hcal_\sigma\); more precisely the line profiles
\(g_\infty^\sharp,g_K^\sharp\sim c^\pm|X|^{-1/2}\) of
\eqref{eq:inner-singular-trace} satisfy
\(|X|^{-1/2}\notin L^2_{\rm loc}(\mathbb R)\supset H^{1/2}_{\rm loc}(\mathbb R)\),
hence
$\displaystyle         \Eedge\cap\Hcal_\sigma=\{0\}$,
and
$\displaystyle         \Dmatch:=\Eedge\oplus\Hcal_\sigma$
is a well-defined direct sum, and the projection \(\Pi_{\rm sing}\) onto the
\(\Eedge\)-component is well defined.
\end{lemma}
\begin{proof}
\(\int_{0}^{1}|X|^{-1}\,dX=\infty\), so \(|X|^{-1/2}\) is not locally square
integrable on the line; a fortiori it does not belong to
\(H^{1/2}_\sigma(\mathbb R)\) or \(H^{-1/2}\cap L^2\)-regular classes used in
\eqref{eq:Hsigma}.  (Note that the corresponding bulk field
\(r^{-1/2}\Vm\) is locally square integrable in two dimensions; the
exclusion is genuinely a trace-level statement, which is why \(\Fsing\) is
recorded as a line datum in \eqref{eq:inner-singular-trace}.)  Since
\(\Fsing\neq0\) has zero bulk components and non-\(\Hcal_\sigma\) line
components, \(\Eedge\cap\Hcal_\sigma=\{0\}\).
\end{proof}
\begin{hypothesis}[Fredholm lower-deck structure with augmented
solvability]
\label{hyp:Fredholm-LD}
For each fixed \(\Omega\) in \eqref{eq:distinguished-LD-frequency}:
\begin{enumerate}
\item \(\LTD(\Omega):\Xcal_\sigma\to\Hcal_\sigma\) is Fredholm of index zero,
\begin{equation}
\label{eq:Fredholm-index}
        \ind\LTD(\Omega)=0,
\end{equation}
and
\begin{equation}
\label{eq:adjoint-kernel-one}
        \ker \LTD(\Omega)^\ast
        =
        \spn\{\Psi^\ast(\Omega)\},\qquad
        \Psi^\ast=(u^\ast,q,b,\mu) ;
\end{equation}
\item (augmented solvability) for every \(A\in\mathbb C\) there exists a
field \(W_A\) in the augmented graph class associated with
\(\Dmatch=\Eedge\oplus\Hcal_\sigma\) such that
\(\LTD(\Omega)W_A=F_{\rm match}(A)\), and the Green identity
\eqref{eq:app-Green} holds for the pair \((W_A,\Psi^\ast)\) with finite edge
concomitant.
\end{enumerate}
\end{hypothesis}
Part (ii) is the precise statement needed to read the compatibility relation
below as an identity in \(A\); it is not implied by part (i), and in
the abstract setting it is part of limitation (i) of \Cref{sec:conclusion}.  In the linear-shear model
of \Cref{sec:model} it holds automatically: the Wiener--Hopf construction
produces solutions for both the Kutta and the non-Kutta edge
normalizations, the latter realizing exactly the singular class \(\Eedge\)
(this is the classical polynomial ambiguity of the entire
function~\cite{orszag1970instability,crighton1985kutta}).
\begin{remark}[Interpretation of the kernel and cokernel]
\label{rem:kernel-interpretation}
If \(\LTD(\Omega)\) were invertible, the Fredholm condition would be vacuous
and no adjoint selection would occur.  The weight \eqref{eq:weights-new} is
chosen precisely so that this does not happen: the inner counterpart of the
shed wake mode is admitted by the downstream weight and furnishes a kernel
element, \(\dim\ker\LTD(\Omega)\geq1\); index zero then forces a cokernel of
equal dimension, and \eqref{eq:adjoint-kernel-one} asserts that no further
degeneracy occurs.  Physically, the cokernel functional \(\Psi^\ast\) measures
resonant forcing of the shed wake mode, and
\(\langle\FKH,\Psi^\ast\rangle\neq0\) states that the wake-mode matching data
force their own resonance.  The kernel element accounts for the expected inner
non-uniqueness: the amplitude of the shed mode is not determined by the inner
problem alone but by the matching constraint, which is the content of
\Cref{thm:conditional-viscous-Kutta}.  In the model of \Cref{sec:model} all of
this is explicit.
\end{remark}
By the Fredholm alternative,
$\displaystyle         \Ran\LTD(\Omega)
        =
        \{F\in\Hcal_\sigma:
        \langle F,\Psi^\ast(\Omega)\rangle_{\Hcal_\sigma}=0\}$.
Consequently, for the regular forcing in \eqref{eq:LTD-regular-forced},
\begin{equation}
\label{eq:regular-Fredholm-condition}
        \big\langle
        \Finc(\Omega)+A\FKH(\Omega),\Psi^\ast(\Omega)
        \big\rangle_{\Hcal_\sigma}=0,
\end{equation}
and therefore, provided \(\langle\FKH,\Psi^\ast\rangle\neq0\),
$\displaystyle         A_{\rm Fr}(\Omega)
        =
        -
        \frac{
        \langle \Finc(\Omega),\Psi^\ast(\Omega)\rangle_{\Hcal_\sigma}}
        {
        \langle \FKH(\Omega),\Psi^\ast(\Omega)\rangle_{\Hcal_\sigma}} $.
% ---------------------------------------------------------------------
\subsection{Edge singular trace and concomitant}
\label{subsec:edge-concomitant}
By Lemma~\ref{lem:trace-exclusion} the matching data decompose as
\[
        F_{\rm match}(A)
        =
        C_-(A)\Fsing+\Freg(A),
        \qquad
        \Freg(A)=\Finc+A\FKH\in\Hcal_\sigma ,
\]
and the singular coefficient is recovered by
$\displaystyle         \Pi_{\rm sing}F_{\rm match}(A)=C_-(A)\Fsing$ .
The boundary concomitant \eqref{eq:LD-concomitant}, evaluated on a small
edge contour \(\partial B_\rho^+\) and paired with the adjoint state,
together with the singular trace
\eqref{eq:inner-singular-trace} against \((b,\mu)\),
defines a finite edge functional
$\displaystyle         \Bedge:\Eedge\times\ker\LTD(\Omega)^\ast\to\mathbb C$ such that
\[
        \Bedge(G,\Psi^\ast)
        =
        \fp\lim_{\rho\downarrow0}
        \int_{\partial B_\rho^+}
        \big(J^X(G,\Psi^\ast)n_X+J^Y(G,\Psi^\ast)n_Y\big)\,ds
        +
        \fp\int_{\mathbb R}
        \big(g_K^\sharp\,b+g_\infty^\sharp\,\mu\big)\,dX ,
\]
the second term being the trace-level form used in practice (and in
\Cref{sec:model}); see Appendix~\ref{app:formal-adjoint}.  For
\(G=C\,r^{-1/2}\Vm\), linearity gives
$\displaystyle        \Bedge(C\,r^{-1/2}\Vm,\Psi^\ast)
        =
        C\,\Bedge(r^{-1/2}\Vm,\Psi^\ast)$.
\begin{hypothesis}[Edge-concomitant nondegeneracy]
\label{hyp:edge-concomitant}
For the adjoint generator in \eqref{eq:adjoint-kernel-one},
$\displaystyle         \kappa(\Omega)
        :=
        \Bedge(r^{-1/2}\Vm,\Psi^\ast(\Omega))
        \neq0 $.
Then,
\begin{equation}
\label{eq:Bedge-Cminus}
        \Bedge(C_-(A)r^{-1/2}\Vm,\Psi^\ast(\Omega))
        =
        C_-(A)\kappa(\Omega).
\end{equation}
\end{hypothesis}
The full compatibility relation for data in
\(\Dmatch=\Eedge\oplus\Hcal_\sigma\), obtained from the Green identity
\eqref{eq:app-Green} under Hypothesis~\ref{hyp:Fredholm-LD}(ii), is
\begin{equation}
\label{eq:full-compatibility}
        \Bedge(C_-(A)r^{-1/2}\Vm,\Psi^\ast)
        +
        \langle \Finc+A\FKH,\Psi^\ast\rangle_{\Hcal_\sigma}
        =
        0.
\end{equation}
Using \eqref{eq:Bedge-Cminus},
\begin{equation}
\label{eq:compat-expanded}
        C_-(A)\kappa(\Omega)
        +
        \langle \Finc(\Omega)+A\FKH(\Omega),\Psi^\ast(\Omega)
        \rangle_{\Hcal_\sigma}
        =
        0 .
\end{equation}
% ---------------------------------------------------------------------
\subsection{Reduction of the Kutta--Fredholm consistency}
\label{subsec:KF-noncirc}
Relation \eqref{eq:compat-expanded} is an identity in \(A\) precisely
because of the augmented solvability Hypothesis~\ref{hyp:Fredholm-LD}(ii): the Green
pairing holds for the whole affine family, not only for the selected value of
\(A\).  It therefore determines the adjoint pairing of the regular data
a priori in terms of the edge concomitant, which removes the apparent
need to impose the consistency relation \eqref{eq:intro-KF-consistency} as a
separate hypothesis.
\begin{proposition}[Reduction of the Kutta--Fredholm hypotheses]
\label{prop:KF-noncirc}
Assume Hypothesis~\ref{hyp:Fredholm-LD} (both parts) and \(\kappa(\Omega)\neq0\).
Then, for every \(A\in\mathbb C\),
\begin{equation}
\label{eq:KF-identity}
        \langle \Finc(\Omega)+A\FKH(\Omega),\Psi^\ast(\Omega)\rangle_{\Hcal_\sigma}
        =
        -\kappa(\Omega)\,C_-(A).
\end{equation}
In particular, matching the affine coefficients in \(A\),
\begin{equation}
\label{eq:KF-coeff-match}
        \langle \Finc,\Psi^\ast\rangle_{\Hcal_\sigma}=-\kappa\,C_-^{(0)},
        \qquad
        \langle \FKH,\Psi^\ast\rangle_{\Hcal_\sigma}=-\kappa\,C_-^{(KH)} .
\end{equation}
Consequently:
\begin{enumerate}
\item the consistency relation \eqref{eq:intro-KF-consistency} holds with
\(\chi(\Omega)=-\kappa(\Omega)\neq0\);
\item the nondegeneracy \(\langle\FKH,\Psi^\ast\rangle\neq0\) holds if and only
if \(C_-^{(KH)}\neq0\);
\item the two scalar selection conditions coincide:
\(C_-(A)=0\Leftrightarrow\langle\Finc+A\FKH,\Psi^\ast\rangle=0\).
\end{enumerate}
Thus hypothesis {\rm(H4)} of \Cref{thm:intro-main} is not independent of the
others: given the Fredholm realization and the augmented solvability of
{\rm(H3)}, it reduces to the single nondegeneracy \(\kappa(\Omega)\neq0\)
together with \(C_-^{(KH)}\neq0\) from {\rm(H2)}.
\end{proposition}
\begin{proof}
By Hypothesis~\ref{hyp:Fredholm-LD}(ii), for each \(A\) there is \(W_A\) in the
augmented class with \(\LTD(\Omega)W_A=F_{\rm match}(A)\) and a valid Green
identity against \(\Psi^\ast\in\ker\LTD(\Omega)^\ast\).  Since
\(\LTD(\Omega)^\ast\Psi^\ast=0\) and the admissible line/decay terms vanish,
the only remaining boundary term is the finite edge concomitant, giving
\eqref{eq:full-compatibility} for that \(A\); as this holds for every
\(A\in\mathbb C\), substituting \eqref{eq:Bedge-Cminus} yields
\eqref{eq:compat-expanded} identically in \(A\), i.e.\
\eqref{eq:KF-identity}.  Because both sides of \eqref{eq:KF-identity} are
affine in \(A\) and \(C_-(A)=C_-^{(0)}+A C_-^{(KH)}\), matching coefficients
gives \eqref{eq:KF-coeff-match}.  Claims (i)--(iii) are immediate: (i) is
\eqref{eq:KF-identity}; (ii) follows from the second equation in
\eqref{eq:KF-coeff-match} and \(\kappa\neq0\); (iii) follows from
\eqref{eq:KF-identity} and \(\kappa\neq0\).
\end{proof}
\begin{remark}
\label{rem:KF-noncirc}
Proposition~\ref{prop:KF-noncirc} resolves a potential circularity: one need not posit
both the regular Fredholm solvability \eqref{eq:regular-Fredholm-condition}
and the consistency \eqref{eq:intro-KF-consistency} as separate scalar
constraints on the single amplitude \(A\).  The genuine analytic content is
concentrated in (a) the Fredholm realization with augmented solvability,
Hypothesis~\ref{hyp:Fredholm-LD} (limitation (i) of \Cref{sec:conclusion} for the true base flow; automatic in
the model of \Cref{sec:model}), and (b) the edge-concomitant nondegeneracy
\(\kappa\neq0\) of Hypothesis~\ref{hyp:edge-concomitant} (limitation (ii) of \Cref{sec:conclusion}; a theorem in
the model, Proposition~\ref{prop:model-H4}).  All algebraic results then follow.
\end{remark}
% ---------------------------------------------------------------------
\subsection{Viscous derivation of \texorpdfstring{\(C_-(A)=0\)}{C-(A)=0}}
\label{subsec:viscous-Kutta-derivation}
The bounded lower-deck class excludes the singular component:
\[
        W\in\Xcal_\sigma
        \quad\Longrightarrow\quad
        \LTD(\Omega)W\in\Hcal_\sigma
        \quad\Longrightarrow\quad
        \Pi_{\rm sing}\LTD(\Omega)W=0
\]
by Lemma~\ref{lem:trace-exclusion}.  Thus a bounded viscous--inviscid matching
solution requires
\[
\label{eq:Cminus-zero-necessary}
        \Pi_{\rm sing}F_{\rm match}(A)=0
        \quad\Longleftrightarrow\quad
        C_-(A)=0 .
\]
Once \(C_-(A)=0\), the remaining forcing lies in \(\Hcal_\sigma\), and
Fredholm solvability is exactly
$\displaystyle         \langle \Finc+A\FKH,\Psi^\ast\rangle_{\Hcal_\sigma}=0 $.
\begin{theorem}[Kutta selection as Fredholm compatibility]
\label{thm:conditional-viscous-Kutta}
Assume Hypotheses~\ref{hyp:Fredholm-LD} and~\ref{hyp:edge-concomitant} and
\(C_-^{(KH)}\neq0\).  Then bounded lower-deck matching to the outer family
\(\Phi_A^{\rm out}\) is possible only if
$        C_-(A)=0 $.
Equivalently,
$\displaystyle         A=A_{\rm Kutta}
        =
        -\frac{C_-^{(0)}}{C_-^{(KH)}}$ .
For this value of \(A\), the regular lower-deck problem is solvable iff
$\displaystyle         \big\langle
        \Finc(\Omega)+A_{\rm Kutta}\FKH(\Omega),
        \Psi^\ast(\Omega)
        \big\rangle_{\Hcal_\sigma}=0 $.
Hence the Kutta-selected and Fredholm-selected amplitudes coincide:
\begin{equation}
\label{eq:Arec-inner}
        \Arec(\Omega)
        =
        A_{\rm Kutta}
        =
        A_{\rm Fr}(\Omega)
        =
        -
        \frac{
        \langle \Finc(\Omega),\Psi^\ast(\Omega)\rangle_{\Hcal_\sigma}}
        {
        \langle \FKH(\Omega),\Psi^\ast(\Omega)\rangle_{\Hcal_\sigma}} .
\end{equation}
\end{theorem}
\begin{proof}
The matching datum has the decomposition
\[
        F_{\rm match}(A)=C_-(A)\Fsing+\Finc+A\FKH,
        \qquad
        \Fsing\in\Eedge,\quad \Finc+A\FKH\in\Hcal_\sigma .
\]
Since \(\Ran\LTD(\Omega)\subset\Hcal_\sigma\) and
\(\Eedge\cap\Hcal_\sigma=\{0\}\) (Lemma~\ref{lem:trace-exclusion}), bounded
matching implies \(C_-(A)=0\).  By
\(C_-(A)=C_-^{(0)}+AC_-^{(KH)}\) and \(C_-^{(KH)}\neq0\),
$\displaystyle         A=-C_-^{(0)}/C_-^{(KH)}$.
For this value the singular part vanishes; by Proposition~\ref{prop:KF-noncirc} the
nondegeneracy \(\langle\FKH,\Psi^\ast\rangle=-\kappa C_-^{(KH)}\neq0\) holds, and
the Fredholm alternative gives
\[
        \Finc+A\FKH\in\Ran\LTD(\Omega)
        \Longleftrightarrow
        \langle \Finc+A\FKH,\Psi^\ast\rangle_{\Hcal_\sigma}=0 .
\]
Solving the scalar equation gives \eqref{eq:Arec-inner}.  The concomitant
hypothesis gives the equivalent edge form
\[
        C_-(A)=0
        \Longleftrightarrow
        \Bedge(C_-(A)r^{-1/2}\Vm,\Psi^\ast)=0,
\]
because \(\kappa(\Omega)\neq0\).
\end{proof}
Combining the preceding identities,
$\displaystyle         C_-(A)=0$
iff
$\displaystyle         \Pi_{\rm sing}F_{\rm match}(A)=0$
iff
$\displaystyle      \Finc+A\FKH\in\Ran\LTD(\Omega)$
and, under Hypothesis~\ref{hyp:Fredholm-LD},
$\displaystyle         \Finc+A\FKH\in\Ran\LTD(\Omega)$ iff
$\displaystyle         \langle \Finc+A\FKH,\Psi^\ast\rangle_{\Hcal_\sigma}=0 $.
Thus the unsteady Kutta condition is the Fredholm compatibility condition of
the viscous lower deck, and the selected receptivity amplitude is the adjoint
quotient \eqref{eq:Arec-inner}.
% =====================================================================
\section{Exact verification in the linear-shear lower-deck model}
\label{sec:model}
This section verifies the inner hypotheses {\rm(H3)}--{\rm(H4)} of
\Cref{thm:intro-main} in closed form for the canonical linear-shear model of
the unsteady lower deck.  The model retains exactly the two features on which
the selection mechanism rests---the plate/wake switching of the boundary
condition at \(X=0\) and the pressure--displacement interaction---while
freezing the base flow at its uniform-shear profile.  It is the trailing-edge
analogue of Terent'ev's vibrating-ribbon problem~\cite{terent1981linear}, and the
same model underlies the classical lower-branch receptivity
analyses~\cite{goldstein1985scattering,ruban1984generation}.  The result of the section is the
following exact instance of Theorem~\ref{thm:intro-main}.
\begin{theorem}[Exact model selection mechanism]
\label{thm:model-anchor}
For the linear-shear lower-deck model \eqref{eq:model-base}, away from the
discrete resonance set \(\Sigma\) of Proposition~\ref{prop:model-H4}, the Kutta
amplitude selected by exclusion of the edge singularity coincides with the
adjoint Fredholm quotient and with the downstream pole residue:
$\displaystyle         \Arec^{\rm m}(\Omega)
        =
        -\frac{C_-^{(0)}}{C_-^{(KH)}}
        =
        -\frac{\langle \Finc,\Psi^\ast_{\rm m}\rangle_{\Hcal_\sigma}}
              {\langle \FKH,\Psi^\ast_{\rm m}\rangle_{\Hcal_\sigma}}
        =
        i\Res_{\alpha=\alpha_w}\Mcal^{\rm m}(\alpha;\Omega)$,
with 
$\displaystyle         \alpha_w=\Omega^{1/2}$.
Moreover the adjoint field is generated mode-wise by the Airy pair
\begin{equation}
\label{eq:model-anchor-adjoint}
        u^\ast(Y)=\Ai'\big(z(Y;\alpha)\big),
        \qquad
        q(Y)=\frac{c(\alpha)^2}{i\alpha}\,\Ai\big(z(Y;\alpha)\big),
\end{equation}
dual to the primal shear structure \(u_Y\propto\Ai(z)\).
\end{theorem}
\begin{proof}
Combine \Cref{thm:model-adjoint} (adjoint structure, giving
\eqref{eq:model-anchor-adjoint}), Proposition~\ref{prop:model-fredholm} (Wiener--Hopf
reduction, Fredholm structure, and augmented solvability),
Proposition~\ref{prop:model-H4} (nondegeneracy of the edge concomitant, defining
\(\Sigma\)), and the general
\Cref{thm:conditional-viscous-Kutta,thm:flat-pole-residue}; the details
occupy \Cref{subsec:model-primal,subsec:model-adjoint,subsec:model-kappa} and
are collected in Corollary~\ref{cor:model-main}.
\end{proof}
% ---------------------------------------------------------------------
\subsection{The model operator}
\label{subsec:model-setup}
Take, in \eqref{eq:lin-LD},
\begin{equation}
\label{eq:model-base}
        U_0=\lambda_0Y,\qquad V_0=0,\qquad \lambda_0>0,
\end{equation}
with the symmetric-component boundary conditions
\eqref{eq:LD-bc}, \eqref{eq:LD-bc-sym} and the flat-plate interaction law
\eqref{eq:LD-interaction}.  Write \(\LTD^{\rm m}(\Omega)\) for the resulting
operator on the weighted spaces of \Cref{subsec:weighted-Fredholm}.  The true
trailing-edge base state differs from \eqref{eq:model-base} by smooth \(\mathcal O(1)\)
terms (displacement, wake centerline acceleration).
Throughout, Fourier transforms are
\(\widehat h(\alpha)=\int h(X)e^{-i\alpha X}dX\), primal modes are
proportional to \(e^{i\alpha X}\), adjoint modes to \(e^{-i\alpha X}\)
(bilinear pairing), and
\begin{equation}
\label{eq:model-z}
        c(\alpha):=(i\alpha\lambda_0)^{1/3},
        \qquad
        z(Y;\alpha):=c(\alpha)Y+z_0(\alpha),
        \qquad
        z_0(\alpha):=-\,\frac{i\Omega}{c(\alpha)^2},
\end{equation}
with the principal branch of \((i\alpha\lambda_0)^{1/3}\) cut along
\(\alpha\in i[0,\infty)\), so that \(|\arg c|\le\pi/6\) for real \(\alpha\)
and \(\Ai(z)\to0\) as \(Y\to+\infty\).  Also set
$\displaystyle         \kappa_1(z_0):=\int_{z_0}^{\infty}\Ai(s)\,ds$ ,
and let \(\gamma(\alpha)\) be the analytic continuation of \(|\alpha|\) with
cuts on the imaginary axis, then \(\gamma(\alpha)=|\alpha|\) for
\(\alpha\in\mathbb R\).
% ---------------------------------------------------------------------
\subsection{Primal Airy structure, impedances, and the wake pole}
\label{subsec:model-primal}
For a primal mode \((u,v,p,a)=(f(Y),\hat v(Y),\hat p,\hat a)e^{i\alpha X}\),
elimination of \(\hat v\) and \(\hat p\) by cross-differentiation of
\eqref{eq:lin-LD} with \eqref{eq:model-base} gives
$\displaystyle         f'''=(i\alpha\lambda_0Y-i\Omega)f'
        =c^2\,z\,f' $,
so that \(f'\) satisfies the Airy equation in \(z\); the decaying branch is
\begin{equation}
\label{eq:model-fprime}
        f'(Y)=B\,\Ai(z),\qquad B\in\mathbb C .
\end{equation}
Two half-line impedances follow.
\emph{Plate} (\(u(0)=v(0)=0\)): integrating \eqref{eq:model-fprime} with
\(f(0)=0\) and evaluating the momentum equation at \(Y=0\),
\[
        i\alpha\hat p=f''(0)=Bc\Ai'(z_0),
        \qquad
        \hat a=\frac{B}{c}\,\kappa_1(z_0),
        \qquad
        Z_p(\alpha;\Omega):=\frac{\hat p}{\hat a}
        =
        \frac{c^2\,\Ai'(z_0)}{i\alpha\,\kappa_1(z_0)} ,
\]
the classical lower-branch impedance~\cite{terent1981linear,sychev1998asymptotic}.
Wake (symmetric: \(u_Y(0)=v(0)=0\)): decay of \(u_Y\) forces
\(B=0\) unless \(\Ai(z_0)=0\); the generic wake mode is therefore the
shear-free slug
\begin{equation}
\label{eq:model-slug}
        f\equiv\hat a,\qquad
        \hat v=-i\alpha\hat a\,Y,\qquad
        -i\Omega\hat a+i\alpha\hat p=0,
\end{equation}
the convective term \(i\alpha\lambda_0Y\hat a\) being cancelled exactly by
\(\lambda_0\hat v\).  Hence
\begin{equation}
\label{eq:model-Zw}
        Z_w(\alpha;\Omega):=\frac{\hat p}{\hat a}=\frac{\Omega}{\alpha}.
\end{equation}
Define the plate and wake dispersion functions
\[
\label{eq:model-Dp-Dw}
        D_p(\alpha;\Omega):=Z_p(\alpha;\Omega)-\gamma(\alpha),
        \qquad
        D_w(\alpha;\Omega):=\Omega-\alpha\gamma(\alpha).
\]
Zeros of \(D_p\) are the lower-branch Tollmien--Schlichting modes of the
semi-infinite plate; zeros of \(D_w\) are the wake modes of the model.  For
\(\Omega>0\),
\begin{equation}
\label{eq:model-wakepole}
        D_w(\alpha_w;\Omega)=0,\qquad
        \alpha_w=\Omega^{1/2}>0,\qquad
        \partial_\alpha D_w(\alpha_w;\Omega)=-2\Omega^{1/2}\neq0 :
\end{equation}
a simple, neutral wake pole.  Its causal classification follows
Briggs--Bers: continuing \(\Omega\mapsto\Omega+i\varsigma\),
\(\varsigma>0\), gives \(\alpha_w=(\Omega+i\varsigma)^{1/2}\) with
\(\Im\alpha_w>0\), so the pole descends onto the real axis from above as
\(\varsigma\downarrow0\) and belongs to the downstream set
\(\mathcal P_{\rm down}\) of \Cref{subsec:flat-residue}.  The symmetric model
wake mode is thus the neutral limiting case of the convectively unstable
situation \(\Im\alpha_{KH}<0\) of Assumption~\ref{ass:simple-KH}; 
see Remark~\ref{rem:model-antisym} for the antisymmetric (flapping) component.

Introduce the half-line unknowns
$\displaystyle         \tau(X):=u_Y(X,0)\,\mathbf 1_{X<0}$,
and
$\displaystyle         U_c(X):=u(X,0)\,\mathbf 1_{X>0}$,
whose transforms \(\widehat\tau_-\), \(\widehat U_{c+}\) are analytic in the
upper and lower half-planes respectively. 
Solving the \(Y\)-problem for arbitrary
\((\widehat\tau,\widehat U_c)\) and eliminating \((\hat p,\hat a)\) with the
interaction law \(\hat p=\gamma(\alpha)\hat a+\widehat f_K\) gives the scalar
Wiener--Hopf equation
\begin{equation}
\label{eq:model-WH}
        D_w(\alpha;\Omega)\,\widehat U_{c+}(\alpha)
        =
        -\,\frac{\alpha\,\kappa_1(z_0)}{c\,\Ai(z_0)}\,
        D_p(\alpha;\Omega)\,\widehat\tau_-(\alpha)
        +
        \alpha\,\widehat f(\alpha),
\end{equation}
where \(\widehat f\) collects the transformed matching data.  The kernel is
\begin{equation}
\label{eq:model-kernel}
        K(\alpha;\Omega)
        =
        -\,\frac{\alpha\,\kappa_1(z_0(\alpha))\,D_p(\alpha;\Omega)}
                {c(\alpha)\,\Ai(z_0(\alpha))\,D_w(\alpha;\Omega)} ,
\end{equation}
meromorphic off the imaginary-axis cuts, with large-\(\alpha\) behavior
\(K=\mathcal O(\alpha^{-1/3})\) determined by \(z_0\to0\), \(D_p\sim-\gamma\),
\(D_w\sim-\alpha\gamma\); the canonical factorization \(K=K_+K_-\) with
zero-free factors exists on any horizontal contour avoiding the zeros of
\(D_p\), \(D_w\), \(\Ai(z_0)\), and \(\kappa_1(z_0)\), with index tracking
fixed by the winding number of \(K\) along the weighted contour
(cf.~\cite{noble1962methods}).  The Kutta-normalized response (minimal edge growth;
\Cref{subsec:WH-representation}) is the meromorphic function
$\displaystyle         \Mcal^{\rm m}(\alpha;\Omega)
        =
        \frac{\Ncal^{\rm m}(\alpha;\Omega)}{D_w(\alpha;\Omega)}$,
with \(\Ncal^{\rm m}\) explicit in terms of \(K_\pm\) and the data, and the
selected wake amplitude of \Cref{thm:flat-pole-residue} is, by
\eqref{eq:model-wakepole},
\begin{equation}
\label{eq:model-residue}
        \Arec^{\rm m}(\Omega)
        =
        i\Res_{\alpha=\alpha_w}\Mcal^{\rm m}
        =
        -\,\frac{i\,\Ncal^{\rm m}(\Omega^{1/2};\Omega)}{2\,\Omega^{1/2}} .
\end{equation}
% ---------------------------------------------------------------------
\subsection{The Airy adjoint}
\label{subsec:model-adjoint}
The central computation of this section is that the adjoint system
\eqref{eq:formal-adjoint} is also exactly solvable, with a basis dual to the
primal Airy structure.
\begin{theorem}[Airy structure of the adjoint]
\label{thm:model-adjoint}
For \(U_0=\lambda_0Y\), \(V_0=0\), an adjoint mode
\((u^\ast,q)=(g(Y),\hat q(Y))e^{-i\alpha X}\) of \eqref{eq:formal-adjoint}
satisfies the reduced third-order equation
\begin{equation}
\label{eq:model-adjoint-ode}
        g'''-(i\alpha\lambda_0Y-i\Omega)\,g'-2i\alpha\lambda_0\,g=0,
        \qquad\text{i.e.}\qquad
        g_{zzz}-z\,g_z-2g=0
\end{equation}
in the variable \(z\) of \eqref{eq:model-z}, and the adjoint pressure is
recovered algebraically as
\begin{equation}
\label{eq:model-qhat}
        \hat q
        =
        \frac{c^2}{i\alpha}\,h(z),
        \qquad
        h:=g_{zz}-z\,g,\qquad h_z=g .
\end{equation}
A fundamental system of \eqref{eq:model-adjoint-ode} is
\begin{equation}
\label{eq:model-adjoint-basis}
        g\in\spn\{\Ai'(z),\ \Bi'(z),\ \Gi'(z)\},
\end{equation}
with companions \(h\in\spn\{\Ai(z), \Bi(z), \Gi(z)\}\) respectively, where
\(\Gi\) is the Scorer function, \(\Gi''(z)-z\Gi(z)=-1/\pi\)
\cite{apostol2010nist}.  The solutions admissible as \(Y\to+\infty\) are spanned by
the recessive pair and the algebraically decaying Scorer pair,
\begin{equation}
\label{eq:model-adjoint-admissible}
        (g,\hat q)\in
        \spn\Big\{
        \big(\Ai'(z),\,\tfrac{c^2}{i\alpha}\Ai(z)\big),\
        \big(\Gi'(z),\,\tfrac{c^2}{i\alpha}\Gi(z)\big)
        \Big\},
\end{equation}
and the adjoint line states are finite and explicit; in particular, for the
recessive component,
\begin{equation}
\label{eq:model-bbar}
        \bar U^\ast(X)
        =
        \int_0^\infty u^\ast\,dY
        =
        -\,\frac{\Ai(z_0)}{c}\,e^{-i\alpha X},
        \qquad
        \widehat b\propto i\alpha\,\frac{\Ai(z_0)}{c} .
\end{equation}
\end{theorem}
\begin{proof}
With \eqref{eq:model-base} the adjoint system \eqref{eq:formal-adjoint} reads
\(-i\Omega u^\ast-\lambda_0Yu^\ast_X-u^\ast_{YY}-q_X=0\),
\(q_Y=\lambda_0u^\ast\).  Inserting \((g,\hat q)e^{-i\alpha X}\),
\[
        -i\Omega g+i\alpha\lambda_0Yg-g''+i\alpha\hat q=0,
        \qquad
        \hat q'=\lambda_0 g .
\]
Solving the first relation for \(i\alpha\hat q=g''+i\Omega g-i\alpha\lambda_0Yg\)
and differentiating, the second relation gives
\(i\alpha\lambda_0g=g'''+i\Omega g'-i\alpha\lambda_0g-i\alpha\lambda_0Yg'\),
which is \eqref{eq:model-adjoint-ode}; passing to \(z\)-units uses
\(i\alpha\lambda_0Y-i\Omega=c^2z\) and \(i\alpha\lambda_0=c^3\).  For
\eqref{eq:model-qhat}, note
\(i\alpha\hat q=c^2g_{zz}+i\Omega g-(c^2z+i\Omega)g=c^2(g_{zz}-zg)=c^2h\),
and \(h_z=g_{zzz}-g-zg_z=(g_{zzz}-zg_z-2g)+g=g\) by
\eqref{eq:model-adjoint-ode}.  For the basis: if \(g=\Ai'(z)\) then
\(g_z=z\Ai\), \(g_{zz}=\Ai+z\Ai'\), \(g_{zzz}=2\Ai'+z^2\Ai\), and
\(g_{zzz}-zg_z-2g=2\Ai'+z^2\Ai-z^2\Ai-2\Ai'=0\); the same computation holds
verbatim for \(\Bi'\), and for \(\Gi'\) using
\(\Gi''=z\Gi-1/\pi\), the inhomogeneous term cancelling in the combination.
The companions follow from
\(h=g_{zz}-zg\): for \(g=\Ai'\),
\(h=\Ai'''-z\Ai'=(z\Ai)'-z\Ai'=\Ai\), and analogously for the other pairs.
Admissibility: \(\Ai'(z)\) is recessive,
\(\Gi'(z)=\mathcal O(z^{-2})\) and \(\Gi(z)=\mathcal O(z^{-1})\) as \(z\to\infty\) in
\(|\arg z|<\pi/3\)~\cite{apostol2010nist}, so both \((g,\hat q)\) pairs decay, while
\(\Bi'\) grows exponentially and is excluded.  Finally
\(\int_0^\infty\Ai'(z)\,dY=c^{-1}[\Ai(z)]_{z_0}^{\infty}=-c^{-1}\Ai(z_0)\),
giving \eqref{eq:model-bbar} via \(b=\bar U^\ast_X\) from
\eqref{eq:adjoint-interaction-main}.
\end{proof}
\begin{remark}[Primal--adjoint Airy duality]
\label{rem:airy-duality}
The primal solution carries its Airy structure in the shear,
\(u_Y\propto\Ai(z)\), with velocity \(u\) an Airy integral; the adjoint
carries it in the velocity, \(u^\ast\propto\Ai'(z)\), with adjoint
pressure \(q\propto\Ai(z)\).  This is the lower-deck realization of the
familiar duality between direct and adjoint Orr--Sommerfeld structures in
receptivity theory~\cite{goldstein1989boundary}, here in closed form.
\end{remark}
The adjoint half-plane problems mirror the primal ones: on the plate side the
condition \(u^\ast(X,0)=0\) and on the wake side \(u^\ast_Y(X,0)=0\) select
one-parameter combinations of the admissible pair
\eqref{eq:model-adjoint-admissible} for each \(\alpha\), and the homogeneous
adjoint problem reduces, by the same elimination that led to
\eqref{eq:model-WH}, to a scalar Wiener--Hopf problem.
\begin{proposition}[Wiener--Hopf reduction of the model Fredholm problem]
\label{prop:model-fredholm}
Let \(\Sigma_0(\Omega)\) denote the (closed, discrete in \(\Omega\)) set of
real-contour degeneracies, i.e.\ those \(\Omega>0\) for which \(D_p(\cdot;\Omega)\),
\(\Ai(z_0(\cdot))\), or \(\kappa_1(z_0(\cdot))\) vanishes on the weighted
inversion contours of \Cref{subsec:weighted-Fredholm}.  For
\(\Omega\notin\Sigma_0\) and weights \eqref{eq:weights-new} with
\(\vartheta>0\) sufficiently small:
\begin{enumerate}
\item the frozen (translation-invariant) plate and wake limit operators of
\(\LTD^{\rm m}(\Omega)\) are invertible on their weighted lines, their
symbols being governed by \(D_p\) and \(D_w\) respectively; the adjoint frozen
symbols are the transposes and have the same determinants;
\item if the kernel \eqref{eq:model-kernel} admits a canonical
factorization \(K=K_+K_-\) with zero index along the weighted contour---on
which the wake pole \(\alpha_w\) lies above the downstream line
\(\Im\alpha=-\vartheta\)---then
\(\LTD^{\rm m}(\Omega):\Xcal_\sigma\to\Hcal_\sigma\) is Fredholm with
\(\ind\LTD^{\rm m}(\Omega)=0\) and
$\displaystyle         \dim\ker\LTD^{\rm m}(\Omega)=\dim\ker\LTD^{\rm m}(\Omega)^\ast=1$ ;
the kernel is generated by the inner wake mode
\(W_w=(e^{i\alpha_wX}\chi_w,\ldots)\) built from the slug
\eqref{eq:model-slug}, and the cokernel by the adjoint Wiener--Hopf state
\(\Psi^\ast_{\rm m}(\Omega)\) assembled from
\eqref{eq:model-adjoint-admissible};
\item augmented solvability (Hypothesis~\ref{hyp:Fredholm-LD}(ii)) holds: for every
\(A\), the Wiener--Hopf construction with the non-Kutta entire-function
normalization produces a solution of
\(\LTD^{\rm m}(\Omega)W_A=F_{\rm match}(A)\) whose edge behavior realizes the
\(|X|^{-1/2}\) line trace \eqref{eq:inner-singular-trace}, and the Green
identity \eqref{eq:app-Green} holds with finite edge concomitant.
\end{enumerate}
\end{proposition}
\begin{proof}
(i) is the explicit computation of
\Cref{subsec:model-primal,subsec:model-adjoint}: the frozen plate (resp.\
wake) problem is diagonalized by the Fourier transform with symbol determinant
proportional to \(D_p\) (resp.\ \(D_w\)) after removal of the nonvanishing
factors \(c\,\Ai(z_0)\), \(\kappa_1(z_0)\); transposition does not change
determinants.  For (ii), the operator differs from a direct sum of its frozen
limits by an interval-supported switching term, so the limit-operator
criterion~\cite{rabinovitch2004limit} reduces to the
invertibility in (i); given the assumed zero-index factorization, the index
and the defect dimensions are read off from the standard Wiener--Hopf
argument-principle count~\cite{noble1962methods}, the downstream weight
\(\vartheta>0\) placing \(\alpha_w\) on the kernel side of the contour,
exactly as in Remark~\ref{rem:kernel-interpretation}.  For (iii), the entire function \(P(\alpha)\)
of \Cref{subsec:WH-representation} may be chosen with one extra polynomial
degree; the corresponding solution has the
\(\mathcal (|\alpha|^{-3/2})\) edge decay of
\eqref{eq:Kutta-normalization-transform}, i.e.\ the \(X^{1/2}\) displacement
and \(|X|^{-1/2}\) trace behavior of \(\Eedge\)
\cite{orszag1970instability,crighton1985kutta}; the concomitant integrals converge by the
explicit local exponents.
\end{proof}
The factorization condition in (ii) is verified along the real contour for
all \(\Omega\notin\Sigma_0\), since the kernel \eqref{eq:model-kernel} is
then zero-free there and \(K=\mathcal O(\alpha^{-1/3})\) is compensated by the
standard algebraic prefactor; the excluded set is absorbed into the discrete
set \(\Sigma\) of Proposition~\ref{prop:model-H4}.
% ---------------------------------------------------------------------
\subsection{Nondegeneracy of the edge concomitant}
\label{subsec:model-kappa}
With augmented solvability available, Proposition~\ref{prop:KF-noncirc} applies
unconditionally in the model, and the identity
\(\langle\FKH,\Psi^\ast_{\rm m}\rangle=-\kappa(\Omega)\,C_-^{(KH)}\) converts
the nondegeneracy of \(\kappa\) into the computable statement that the
wake-mode data are not orthogonal to the adjoint state.  The latter pairing is
an evaluation at the wake pole and is explicit.
\begin{proposition}[Model nondegeneracy: {\rm(H4)} holds off a discrete set]
\label{prop:model-H4}
Let \(z_0^w:=z_0(\alpha_w)\).  Then
\begin{equation}
\label{eq:model-z0w}
        z_0^w
        =
        \lambda_0^{-2/3}\,\Omega^{2/3}\,e^{-5i\pi/6},
\end{equation}
so that as \(\Omega\) ranges over \((0,\infty)\) the point \(z_0^w\) traverses
the fixed ray \(\arg z=-5\pi/6\).  The wake-mode pairing evaluates, by
Parseval, at \(\alpha=\alpha_w\):
\begin{equation}
\label{eq:model-pairing}
        \langle\FKH,\Psi^\ast_{\rm m}(\Omega)\rangle_{\Hcal_\sigma}
        =
        \frac{\mathcal C(\Omega)}
             {K_-(\alpha_w;\Omega)\,\partial_\alpha D_w(\alpha_w;\Omega)}\,
        \Ai'\!\big(z_0^w\big),
\end{equation}
where \(\mathcal C(\Omega)\) is a finite product of the normalization
constants of \(\ell_{KH}\), \(\Min\), and the factorization, nonvanishing by
construction.  Consequently
\begin{equation}
\label{eq:model-Sigma}
        \kappa(\Omega)
        =
        -\,\frac{\langle\FKH,\Psi^\ast_{\rm m}\rangle}{C_-^{(KH)}}
        \neq0
        \qquad\text{for all }\Omega\in(0,\infty)\setminus\Sigma,
\end{equation}
where \(\Sigma\supset\Sigma_0\) is discrete.  In particular, under the
factorization condition of Proposition~\ref{prop:model-fredholm}, {\rm(H3)}--{\rm(H4)}
of \Cref{thm:intro-main} hold for the model for all
\(\Omega\in(0,\infty)\setminus\Sigma\).
\end{proposition}
\begin{proof}
Formula \eqref{eq:model-z0w} is direct:
\(z_0^w=-i\Omega(i\Omega^{1/2}\lambda_0)^{-2/3}
=\Omega^{2/3}\lambda_0^{-2/3}e^{-i\pi/2}e^{-i\pi/3}\).
For \eqref{eq:model-pairing}: the matching datum \(\FKH\) of the (normalized)
wake mode is, after the wake-mode subtraction implicit in the weight choice,
concentrated on the downstream half-line with profile
\(\propto e^{i\alpha_wX}\); its bilinear pairing with
\(\Psi^\ast_{\rm m}\) is, by Parseval, the evaluation of the adjoint
transform at \(\alpha=\alpha_w\).  The adjoint transform is the Wiener--Hopf
solution of Proposition~\ref{prop:model-fredholm}(ii); at \(\alpha_w\) it is a product of
(a) the factor \(1/K_-(\alpha_w)\), (b) the simple-zero factor
\(1/\partial_\alpha D_w(\alpha_w)\) arising from the wake-side elimination,
and (c) the recessive adjoint amplitude, which by
\Cref{thm:model-adjoint} is proportional to \(\Ai'(z_0^w)\); collecting the
remaining nonzero normalization constants into \(\mathcal C(\Omega)\) gives
\eqref{eq:model-pairing}.
Nonvanishing: \(K_-(\alpha_w)\neq0\) because the factorization is zero-free
by construction; \(\partial_\alpha D_w(\alpha_w)=-2\Omega^{1/2}\neq0\); and
\(\Ai'(z_0^w)\neq0\) for every \(\Omega>0\), since all zeros of
\(\Ai'\) lie on the negative real axis~\cite{apostol2010nist} while
\(\arg z_0^w=-5\pi/6\neq\pm\pi\).  The only possible degeneracies are those of
\(\mathcal C(\Omega)\) and of the contour conditions defining
\(K_\pm\), i.e.\ zeros of the analytic functions
\(D_p(\cdot;\Omega)\), \(\Ai(z_0(\cdot))\), \(\kappa_1(z_0(\cdot))\) on the
weighted contours, together with possible zeros of \(C_-^{(KH)}\) excluded by
{\rm(H2)}; each is the intersection of the zero set of a nontrivial analytic
function with a fixed ray or contour, hence a discrete set
\(\Sigma\subset(0,\infty)\).  Combining with
Propositions~\ref{prop:KF-noncirc} and~\ref{prop:model-fredholm} yields
\eqref{eq:model-Sigma} and the final claim.
\end{proof}
\Cref{fig:model-airy} illustrates the proposition numerically: along the
wake-pole ray \eqref{eq:model-z0w} the three Airy quantities entering
\eqref{eq:model-pairing} and the kernel \eqref{eq:model-kernel} remain
bounded away from zero over the sampled frequency range.
\begin{figure}[t]
\centering
\includegraphics[width=0.72\textwidth]{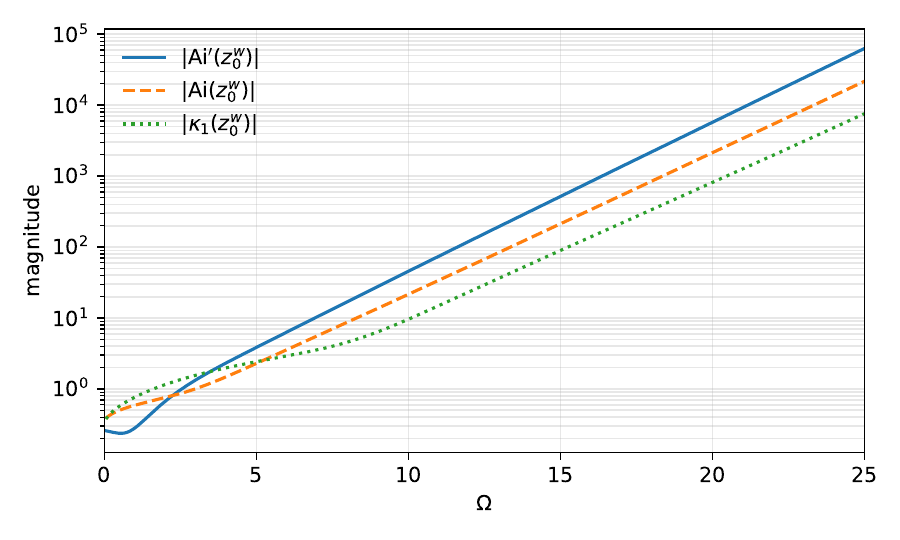}
\caption{Nondegeneracy along the wake-pole ray \eqref{eq:model-z0w}
(\(\lambda_0=1\)): magnitudes of \(\Ai'(z_0^w)\), \(\Ai(z_0^w)\), and
\(\kappa_1(z_0^w)\) for \(\Omega\in(0,25]\).  All three factors remain
bounded away from zero (sampled minima \(\approx0.24\), \(0.39\), and
\(0.38\), respectively), consistent with Proposition~\ref{prop:model-H4}.}
\label{fig:model-airy}
\end{figure}
\clearpage
\begin{corollary}[Inner selection in the model]
\label{cor:model-main}
For the linear-shear model and \(\Omega\in(0,\infty)\setminus\Sigma\), the
conclusions of \Cref{thm:conditional-viscous-Kutta} hold unconditionally:
bounded lower-deck matching enforces \(C_-(A)=0\), the selected amplitude is
the adjoint quotient \eqref{eq:Arec-inner} with
\(\Psi^\ast=\Psi^\ast_{\rm m}\) explicit through
\Cref{thm:model-adjoint}, and it coincides with the wake-pole formula
\eqref{eq:model-residue}.
\end{corollary}
\begin{remark}[Antisymmetric component and the genuine KH pole]
\label{rem:model-antisym}
The computation above is for the symmetric component, whose model wake mode is
the neutral slug \eqref{eq:model-slug}.  For the antisymmetric (flapping)
component, the wake-side conditions \eqref{eq:LD-bc-antisym} replace
\eqref{eq:LD-bc-sym}; the elimination proceeds identically, with
\eqref{eq:model-Zw} replaced by the flapping impedance
\(Z_w^{\rm a}(\alpha;\Omega)\) obtained from the kinematic condition and the
antisymmetric interaction map, and the wake dispersion function
\(D_w^{\rm a}=Z_w^{\rm a}-\gamma^{\rm a}\) acquires complex zeros with
\(\Im\alpha_w^{\rm a}<0\) (the lower-deck counterpart of the
Kelvin--Helmholtz mode), restoring the strict inequality of
Assumption~\ref{ass:simple-KH}.  None of the structural steps changes: the adjoint
basis \eqref{eq:model-adjoint-basis} is the same (it depends only on the bulk
operator), the kernel/cokernel count is the same with \(\vartheta>
|\Im\alpha_w^{\rm a}|\), and Proposition~\ref{prop:model-H4} holds with \(z_0^w\)
evaluated at \(\alpha_w^{\rm a}\), the ray \eqref{eq:model-z0w} replaced by a
curve in \(|\arg z_0^w|<\pi\) still avoiding the negative real axis for
\(\Im\alpha_w^{\rm a}<0\) small.  The explicit form of \(Z_w^{\rm a}\) for the
true near-wake profile requires the numerical base flow
of~\cite{jobe1974numerical} and is left to future work.
\end{remark}

% =====================================================================
\section{Transform representation and pole-residue formula}
\label{sec:transform}
We connect the Fredholm-selected amplitude of \Cref{sec:inner} with the pole
coefficient of the transformed outer problem.  Throughout,
        $\displaystyle \Phi_A^{\rm out}
        =
        \Phi_0^{\rm out}+A\Phi_{KH}^{\rm out}$,
        $\displaystyle C_-(A)=C_-^{(0)}+A C_-^{(KH)}$, and
        $\displaystyle C_-^{(KH)}\neq0 $.
The selected value from the lower deck is
$\displaystyle         \Arec(\Omega)
        =
        -
        \frac{\langle \Finc(\Omega),\Psi^\ast(\Omega)\rangle_{\Hcal_\sigma}}
             {\langle \FKH(\Omega),\Psi^\ast(\Omega)\rangle_{\Hcal_\sigma}} $.
% ---------------------------------------------------------------------
\subsection{Kutta-normalized Wiener--Hopf representation}
\label{subsec:WH-representation}
For the flat plate, introduce
\[
        \widehat f(\alpha)
        =
        \int_{-\infty}^{\infty}f(x)e^{-i\alpha x}\,dx,\qquad
        \widehat f_+(\alpha)=\int_0^\infty f(x)e^{-i\alpha x}\,dx,\qquad
        \widehat f_-(\alpha)=\int_{-\infty}^{0}f(x)e^{-i\alpha x}\,dx .
\]
With \(e^{i\alpha x-i\omega_{\rm phys}t}\), the \(+\)-transform is analytic in
a lower half-strip and the \(-\)-transform in an upper half-strip.  Fourier
transformation of
\[
        \Lout\phi=0,\qquad
        \Lout=\beta^2\partial_x^2+\partial_y^2+2iMk_0\partial_x+k_0^2
\]
gives
$\displaystyle         \partial_y^2\widehat\phi-\mu(\alpha)^2\widehat\phi=0$ with
$\displaystyle         \mu(\alpha)^2=\beta^2\alpha^2+2Mk_0\alpha-k_0^2$ .
The physical branch is fixed by
        $\displaystyle \Re \mu(\alpha)>0$ 
        on the inversion contour, with 
$\displaystyle         \alpha_\pm=\pm\frac{k_0}{1\pm M}$
being the acoustic branch points of \eqref{eq:branch-points}.  Hence
$\displaystyle         \widehat\phi^\pm(\alpha,y)
        =
        \widehat\phi^\pm(\alpha,0)e^{\mp\mu(\alpha)y}$
        for $\pm y>0$.
The plate condition on \(x<0\) and the sheet condition on \(x>0\) yield a
Wiener--Hopf equation~\cite{noble1962methods} of the equation
$\displaystyle         K(\alpha;\omega_{\rm phys},\Gcal)\,\widehat\eta_+(\alpha)
        +
        \widehat q_-(\alpha)
        =
        \widehat f(\alpha;\omega_{\rm phys},\Gcal)$,
where \(K\) is the scalar kernel obtained from the sheet determinant and
\(\widehat f\) is the transformed acoustic forcing.  Its zero set contains the
spatial wake spectrum:
\[
        K(\alpha;\omega_{\rm phys},\Gcal)=0
        \quad\Longleftrightarrow\quad
        \Dcal(\alpha;\omega_{\rm phys},\Gcal)=0
        \quad\text{up to nonzero analytic factors}.
\]
Assume a canonical factorization in a common strip \(\mathfrak S\):
$\displaystyle         K(\alpha;\omega_{\rm phys},\Gcal)
        =
        K_+(\alpha;\omega_{\rm phys},\Gcal)
        K_-(\alpha;\omega_{\rm phys},\Gcal)$,
and
$\displaystyle  K_\pm^{\pm1}\in\mathcal O(\mathfrak S_\pm)$
with the retained downstream wake pole excluded from \(K_-^{-1}\) and kept
explicitly in the meromorphic response.  Splitting
$\displaystyle         \frac{\widehat f}{K_-}
        =
        \left(\frac{\widehat f}{K_-}\right)_+
        +
        \left(\frac{\widehat f}{K_-}\right)_-$
gives
$\displaystyle         K_+\,\widehat\eta_+
        -
        \left(\frac{\widehat f}{K_-}\right)_+
        =
        -\frac{\widehat q_-}{K_-}
        +
        \left(\frac{\widehat f}{K_-}\right)_- $.
The two sides extend to an entire function \(P(\alpha)\).  The edge condition
fixes the polynomial ambiguity:
\begin{equation}
\label{eq:Kutta-normalization-transform}
        C_-(\Phi^{\rm out})=0
        \quad\Longleftrightarrow\quad
        P=P_{\rm Kutta},
\end{equation}
equivalently, in terms of the Abelian correspondence between the edge
behavior of \(\eta\) and the decay of its transform: the non-Kutta state
associated with the \(r^{1/2}\) potential term has
\(\eta\sim{\rm const}\cdot x^{1/2}\) and hence
\(\widehat\eta_+=O(|\alpha|^{-3/2})\) at infinity in \(\mathfrak S_-\), while
the Kutta-normalized state has the attached-sheet behavior
\(\eta\sim{\rm const}\cdot x\,(1+O(x^{1/2}))\) and hence
\[
        \widehat\eta_+(\alpha)
        =
        O(|\alpha|^{-2})
        \qquad (|\alpha|\to\infty\ \text{in }\mathfrak S_-),
\]
the classification of the edge exponents \(x^{1/2}\), \(x\), \(x^{3/2}\)
being that of Orszag \& Crow~\cite{orszag1970instability}; see
also~\cite{crighton1985kutta}.
\begin{assumption}[Kutta-normalized meromorphic response]
\label{ass:WH-meromorphic}
The flat-plate Kutta-normalized outer response has
$\displaystyle         \widehat\eta_+(\alpha)
        =
        \Mcal(\alpha;\omega_{\rm phys},\Gcal)
        =
        \frac{\Ncal(\alpha;\omega_{\rm phys},\Gcal)}
             {\Dcal(\alpha;\omega_{\rm phys},\Gcal)}$
in the downstream deformation domain \(\mathfrak D_-\), where
\(\Ncal,\Dcal\in\mathcal O(\mathfrak D_-)\) except for acoustic cuts, and
\[
        \Dcal(\alpha_{KH};\omega_{\rm phys},\Gcal)=0,\qquad
        \partial_\alpha\Dcal(\alpha_{KH};\omega_{\rm phys},\Gcal)\neq0,\qquad
        \Ncal(\alpha_{KH};\omega_{\rm phys},\Gcal)\neq0 .
\]
\end{assumption}
Thus near \(\alpha_{KH}\),
$\displaystyle        \Mcal(\alpha;\omega_{\rm phys},\Gcal)
        =
        \frac{\mathcal R_{KH}(\omega_{\rm phys},\Gcal)}
             {\alpha-\alpha_{KH}}
        +
        \Mcal_{\rm hol}(\alpha;\omega_{\rm phys},\Gcal)$,
where
\[
\mathcal R_{KH}(\omega_{\rm phys},\Gcal)
        =
        \Res_{\alpha=\alpha_{KH}}\Mcal(\alpha;\omega_{\rm phys},\Gcal)
        =
        \frac{\Ncal(\alpha_{KH};\omega_{\rm phys},\Gcal)}
             {\partial_\alpha\Dcal(\alpha_{KH};\omega_{\rm phys},\Gcal)}.
\]
% ---------------------------------------------------------------------
\subsection{Uniqueness of the Kutta-normalized outer solution}
\label{subsec:kutta-unique-transform}
Let
$\displaystyle         \mathfrak K:\Phi^{\rm out}\mapsto C_-(\Phi^{\rm out})$
be the singular edge functional.  On the affine outer family,
$\displaystyle         \mathfrak K(\Phi_A^{\rm out})
        =
        C_-^{(0)}+A C_-^{(KH)}$ .
Hence
\begin{equation}
\label{eq:A-Kutta-transform}
        A_{\rm Kutta}^{\rm out}
        =
        -\frac{C_-^{(0)}}{C_-^{(KH)}} .
\end{equation}
\begin{lemma}[Uniqueness of the Kutta-normalized outer solution]
\label{lem:kutta-normalized-unique}
Assume
\[
        \ker\Aout^{\rm hom}=\spn\{\Phi_{KH}^{\rm out}\},
        \qquad
        C_-^{(KH)}\neq0 .
\]
Then the set
$\displaystyle         \{\Phi^{\rm out}:\Aout\Phi^{\rm out}=\Fout^{\rm inc},\
        \Phi^{\rm out}\ \text{outgoing},\
        \mathfrak K(\Phi^{\rm out})=0\}$
contains exactly one element, namely
$\displaystyle         \Phi_{\rm Kutta}^{\rm out}
        =
        \Phi_0^{\rm out}
        -
        \frac{C_-^{(0)}}{C_-^{(KH)}}\Phi_{KH}^{\rm out}$.
\end{lemma}
\begin{proof}
Every outgoing forced solution is \(\Phi_A^{\rm out}\).  The constraint
\(\mathfrak K(\Phi_A^{\rm out})=0\) is the scalar equation
\(C_-^{(0)}+A C_-^{(KH)}=0\), which has the unique solution
\eqref{eq:A-Kutta-transform}.
\end{proof}
\begin{corollary}[Identification principle]
\label{cor:identification-principle}
If \(\widetilde\Phi^{\rm out}\) is produced by any transform construction
satisfying the same incident field, outgoing convention, and Kutta
normalization, then
$\displaystyle         \widetilde\Phi^{\rm out}=\Phi_{\rm Kutta}^{\rm out}$.
Consequently, its coefficient of \(\Phi_{KH}^{\rm out}\) is
\(A_{\rm Kutta}^{\rm out}\).
\end{corollary}
Combining this with \Cref{thm:conditional-viscous-Kutta},
\begin{equation}
\label{eq:A-inner-outer-equality}
        A_{\rm Kutta}^{\rm out}
        =
        \Arec(\Omega)
        =
        -
        \frac{
        \langle \Finc(\Omega),\Psi^\ast(\Omega)\rangle_{\Hcal_\sigma}}
        {
        \langle \FKH(\Omega),\Psi^\ast(\Omega)\rangle_{\Hcal_\sigma}} .
\end{equation}
% ---------------------------------------------------------------------
\subsection{Causal inversion and residue formula for the flat plate}
\label{subsec:flat-residue}
The downstream displacement is recovered by
\begin{equation}
\label{eq:inverse-WH}
        \eta(x)
        =
        \frac{1}{2\pi}
        \int_\Gamma
        \Mcal(\alpha;\omega_{\rm phys},\Gcal)e^{i\alpha x}\,d\alpha,
        \qquad x>0 .
\end{equation}
The inversion contour \(\Gamma\) is not a free choice; it is fixed by
causality.  Restore the temporal Laplace by continuing
\(\omega_{\rm phys}\to\omega_{\rm phys}+i\sigma_t\) with \(\sigma_t\to+\infty\)
(the disturbance is switched on at finite time), so that the time transform is
analytic in the upper half \(\omega\)-plane and every spatial pole
\(\alpha_j(\omega_{\rm phys}+i\sigma_t)\) is displaced off the real
\(\alpha\)-axis into a definite half-plane.  With the kernel \(e^{i\alpha x}\),
\(\Gamma\) runs along this causal deformation of the real axis, and the
Briggs--Bers classification~\cite{briggs1964electron,bers1983space,monkewitz1990local}
applies:
\[
        \alpha_j\in\mathcal P_{\rm down}
        \iff
        \Im\alpha_j(\omega_{\rm phys}+i\sigma_t)>0
        \quad\text{for }\sigma_t\gg1 ,
\]
with \(\mathcal P_{\rm up}\) the complementary set.  The label is
invariant under the continuation \(\sigma_t:+\infty\to0^+\), even though the
pole positions move; \(\mathcal P_{\rm down}\) is the causal replacement for the
naive orientation-based set.
For \(x>0\) the kernel decays in the upper half-plane, so \(\Gamma\) is closed
upward and one collects exactly the downstream set:
\begin{equation}
\label{eq:contour-deformation}
        \eta(x)
        =
        i\sum_{\alpha_j\in\mathcal P_{\rm down}}
        \Res_{\alpha=\alpha_j}
        \big(\Mcal(\alpha)e^{i\alpha x}\big)
        +
        \eta_{\rm cuts}(x)+\eta_{\rm arc}(x),
\end{equation}
the constant \(i=\tfrac{1}{2\pi}\cdot2\pi i\) being the counterclockwise
orientation factor.
The Kelvin--Helmholtz pole is critical.  At \(\sigma_t\gg1\) it lies in the
upper half-plane, hence \(\alpha_{KH}\in\mathcal P_{\rm down}\); as
\(\sigma_t\to0^+\) it migrates downward across the real axis to its
physical position \(\Im\alpha_{KH}<0\), dragging \(\Gamma\) below it so that the
pole remains on the downstream side of the contour.  Thus \(\alpha_{KH}\) is
enclosed by the upward closure despite lying in the lower half-plane: it
is a downstream pole that has crossed the axis, which is precisely the
statement that the wake mode is an unstable, spatially
amplifying disturbance carried into \(x>0\),
$\displaystyle         e^{i\alpha_{KH}x}
        =
        e^{i(\Re\alpha_{KH})x}\,e^{|\Im\alpha_{KH}|x}$.
Its contribution is
\begin{equation}
\label{eq:KH-component}
        \eta_{KH}(x)
        =
        A_{\rm pole}(\Omega,\Gcal)\,e^{i\alpha_{KH}x},
        \qquad
        A_{\rm pole}(\Omega,\Gcal)
        =
        i\,\Res_{\alpha=\alpha_{KH}}\Mcal(\alpha;\omega_{\rm phys},\Gcal).
\end{equation}
Since the normalization \eqref{eq:ellKH-def} assigns unit displacement
amplitude to \(\Phi_{KH}^{\rm out}\), \(\eta_{KH}(x)=A\,e^{i\alpha_{KH}x}\)
identifies the coefficient of \(\Phi_{KH}^{\rm out}\) in the inverse-transform
solution as \(A=A_{\rm pole}\), with the orientation factor \(i\) carried
explicitly; no further normalization freedom is invoked.
The deformation, and hence \eqref{eq:KH-component}, presupposes
convective instability: the descending pole \(\alpha_{KH}\) reaches
\(\Im\alpha_{KH}<0\) without colliding with a member of
\(\mathcal P_{\rm up}\).  Such a collision is a Briggs--Bers
pinch~\cite{briggs1964electron,bers1983space}, signals the onset of absolute instability,
and invalidates the simple downstream residue pickup (the fixed-\(x\) response
then grows in \(T\) and is no longer of the form \eqref{eq:KH-component}).  We
assume no pinch throughout; the convective/absolute transition, and any
pole--cut collision with the downstream acoustic branch point
\(\alpha_+=+k_0/(1+M)\), are deferred to \Cref{sec:conclusion}.
If \(\alpha_{KH}\) is simple, then
$\displaystyle         A_{\rm pole}(\Omega,\Gcal)
        =
        \frac{i\,\Ncal(\alpha_{KH};\omega_{\rm phys},\Gcal)}
             {\partial_\alpha\Dcal(\alpha_{KH};\omega_{\rm phys},\Gcal)} $.
Figure~\ref{fig:causal_contour} shows the the integration contour $\Gamma$ in the complex $\alpha$-plane.

\begin{figure}[htbp]
    \centering
    \includegraphics[width=0.85\textwidth]{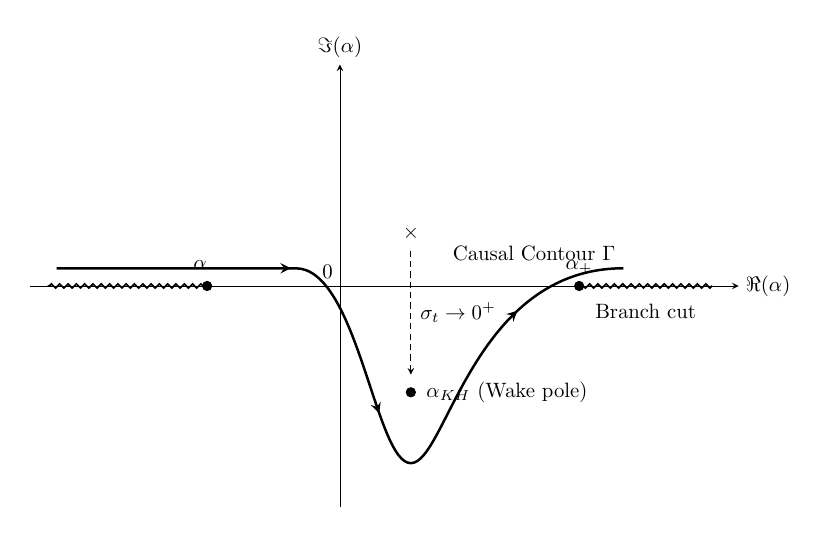}
    \caption{Schematic of the causal integration contour $\Gamma$ in the complex $\alpha$-plane. To satisfy causality as $\sigma_t \to 0^+$, the contour is deformed into the lower half-plane to pass below the moving wake pole $\alpha_{KH}$. The branch points $\alpha_\pm$ and their corresponding branch cuts are also depicted.}
    \label{fig:causal_contour}
\end{figure}
\clearpage
             
\begin{theorem}[Flat-plate pole-residue formula]
\label{thm:flat-pole-residue}
Assume Assumptions~\ref{ass:simple-KH},~\ref{ass:WH-meromorphic}, Hypotheses~\ref{hyp:Fredholm-LD},~\ref{hyp:edge-concomitant}
and \(C_-^{(KH)}\neq0\).  Then, with the wake normalization
\eqref{eq:ellKH-def},
$\displaystyle        \Arec(\Omega,\Gcal)
        =
        -
        \frac{
        \langle \Finc(\Omega),\Psi^\ast(\Omega)\rangle_{\Hcal_\sigma}}
        {
        \langle \FKH(\Omega),\Psi^\ast(\Omega)\rangle_{\Hcal_\sigma}}
        =
        i\Res_{\alpha=\alpha_{KH}}
        \Mcal(\alpha;\omega_{\rm phys},\Gcal)$.
For a simple pole,
$\displaystyle         \Arec(\Omega,\Gcal)
        =
        \frac{i\,\Ncal(\alpha_{KH};\omega_{\rm phys},\Gcal)}
             {\partial_\alpha\Dcal(\alpha_{KH};\omega_{\rm phys},\Gcal)} $.
\end{theorem}
\begin{proof}
By \Cref{thm:conditional-viscous-Kutta}, bounded lower-deck matching selects
the unique Kutta-normalized outer solution.  By
Corollary~\ref{cor:identification-principle}, the Kutta-normalized Wiener--Hopf solution
is the same outer solution.  The causal inverse transform
\eqref{eq:inverse-WH}--\eqref{eq:KH-component} gives the coefficient of the
normalized \(\Phi_{KH}^{\rm out}\) as
\(i\Res_{\alpha=\alpha_{KH}}\Mcal\).  The adjoint quotient is the same
coefficient by \eqref{eq:A-inner-outer-equality}.  The simple-pole expression
follows from
$\displaystyle         \Dcal(\alpha)=
        \partial_\alpha\Dcal(\alpha_{KH})(\alpha-\alpha_{KH})
        +O((\alpha-\alpha_{KH})^2)$.
\end{proof}
The formula is invariant under the rescaling
\[
        \Phi_{KH}^{\rm out}\mapsto c\Phi_{KH}^{\rm out},\qquad
        A\mapsto c^{-1}A,
\]
provided the same rescaling is used in \(\FKH\), \(C_-^{(KH)}\), and
\(\ell_{KH}\); the statement above fixes \(c\) by \eqref{eq:ellKH-def}.  If
\[
        \Dcal(\alpha_{KH})=\Dcal_\alpha(\alpha_{KH})=\cdots
        =\Dcal_{\alpha}^{(m-1)}(\alpha_{KH})=0,\qquad
        \Dcal_{\alpha}^{(m)}(\alpha_{KH})\neq0,
\]
then
$\displaystyle         \Mcal(\alpha)
        =
        \sum_{\ell=1}^{m}
        \frac{\mathcal R_\ell}{(\alpha-\alpha_{KH})^\ell}
        +\Mcal_{\rm hol}(\alpha)$,
and the downstream contribution is
$\displaystyle         \eta_{KH}(x)
        =
        i\,e^{i\alpha_{KH}x}
        \sum_{\ell=1}^{m}
        \frac{(ix)^{\ell-1}}{(\ell-1)!}\mathcal R_\ell$.
Thus the simple quotient \(i\Ncal/\Dcal_\alpha\) is valid only when
\(\alpha_{KH}\) is simple.
The geometry derivative of the simple-pole coefficient is
\begin{equation*}
\begin{aligned}
        \partial_\gamma\Arec
        &=
        i\,\frac{\Ncal_\gamma+\Ncal_\alpha\partial_\gamma\alpha_{KH}}
             {\Dcal_\alpha}
        -
        i\,\frac{\Ncal(\Dcal_{\alpha\gamma}
        +\Dcal_{\alpha\alpha}\partial_\gamma\alpha_{KH})}
             {\Dcal_\alpha^2},\\
        \partial_\gamma\alpha_{KH}
        &=
        -\frac{\Dcal_\gamma}{\Dcal_\alpha},
\end{aligned}
\end{equation*}
all quantities being evaluated at
\((\alpha,\omega_{\rm phys},\Gcal)=(\alpha_{KH},\omega_{\rm phys},\Gcal)\).
Thus, for the flat plate,
       $\displaystyle  \Arec
        =
        i\Res_{\alpha=\alpha_{KH}}\Mcal(\alpha;\omega_{\rm phys},\Gcal)$,
and, together with \Cref{thm:conditional-viscous-Kutta},
\begin{equation}
\label{eq:full-final-identity}
        \text{Fredholm lower-deck compatibility}
        \Longleftrightarrow
        C_-(A)=0
        \Longleftrightarrow
        \text{wake-pole residue selection}.
\end{equation}
The Mellin analogue for finite-angle wedges with self-similar sheet data,
including the wedge pole-residue formula (\Cref{thm:wedge-pole-residue}), is
given in Appendix~\ref{app:wedge}.
% =====================================================================
\section{Discussion and limitations}
\label{sec:conclusion}
We summarize the conditional theory.  Let
\[
        \Phi_A^{\rm out}=\Phi_0^{\rm out}+A\Phi_{KH}^{\rm out},\qquad
        C_-(A)=C_-^{(0)}+A C_-^{(KH)},\qquad C_-^{(KH)}\neq0,
\]
and assume the simple-pole, edge-indicial, Fredholm (with augmented
solvability), edge-concomitant, and Kutta-normalized transform hypotheses:
\[
\begin{gathered}
        \Dcal(\alpha_{KH};\omega_{\rm phys},\Gcal)=0,\qquad
        \Dcal_\alpha(\alpha_{KH};\omega_{\rm phys},\Gcal)\neq0,\qquad
        \Im\alpha_{KH}<0,\\
        \ind\LTD(\Omega)=0,\qquad
        \ker\LTD(\Omega)^\ast=\spn\{\Psi^\ast(\Omega)\},\qquad
        \kappa(\Omega)=\Bedge(r^{-1/2}\Vm,\Psi^\ast(\Omega))\neq0,\\
        \Mcal(\alpha;\omega_{\rm phys},\Gcal)
        =
        \frac{\Ncal(\alpha;\omega_{\rm phys},\Gcal)}
             {\Dcal(\alpha;\omega_{\rm phys},\Gcal)}
        \quad\hbox{near }\alpha_{KH}.
\end{gathered}
\]
Then the selected amplitude is, with the wake normalization
\eqref{eq:ellKH-def},
\[
        \Arec(\Omega,\Gcal)
        =
        -\frac{
        \langle \Finc(\Omega),\Psi^\ast(\Omega)\rangle_{\Hcal_\sigma}}
        {
        \langle \FKH(\Omega),\Psi^\ast(\Omega)\rangle_{\Hcal_\sigma}}
        =
        -\frac{C_-^{(0)}}{C_-^{(KH)}}
        =
        i\Res_{\alpha=\alpha_{KH}}
        \Mcal(\alpha;\omega_{\rm phys},\Gcal).
\]
For a simple pole,
$\displaystyle         \Arec(\Omega,\Gcal)
        =
        \frac{i\,\Ncal(\alpha_{KH};\omega_{\rm phys},\Gcal)}
             {\Dcal_\alpha(\alpha_{KH};\omega_{\rm phys},\Gcal)}$.
Equivalently,
        $C_-(A)=0$ iff
        $\displaystyle \Pi_{\rm sing}F_{\rm match}(A)=0$
iff
       $\displaystyle  \Finc+A\FKH\in\Ran\LTD(\Omega)$
iff
       $\displaystyle  \langle \Finc+A\FKH,\Psi^\ast\rangle_{\Hcal_\sigma}=0 $.
Thus the unsteady Kutta condition is the vanishing of the singular edge trace,
not an additional inviscid boundary condition:
\[
        \Tssing\nabla\phi_A^{\rm out}
        =
        C_-(A)r^{-1/2}\Vm
        =0
        \quad\Longleftrightarrow\quad
        \Bedge(C_-(A)r^{-1/2}\Vm,\Psi^\ast)=0 .
\]
For the linear-shear model of \Cref{sec:model} the inner hypotheses are
theorems (Propositions~\ref{prop:model-fredholm} and~\ref{prop:model-H4}), the adjoint state is the
Airy-derivative field of \Cref{thm:model-adjoint}, and the identities above
hold unconditionally for all \(\Omega>0\) outside a discrete resonance set
(Corollary~\ref{cor:model-main}).
The distinguished frequency variable is
\[
        \Omega
        =
        Re^{-1/4}\frac{\omega_{\rm phys}L}{U},
        \qquad
        \Omega=O(1)
        \Longleftrightarrow
        \frac{\omega_{\rm phys}L}{U}=O(Re^{1/4}).
\]
Hence the leading receptivity law has the reduced form
       $\displaystyle  \Arec
        =
        \Arec\!\left(
        Re^{-1/4}\frac{\omega_{\rm phys}L}{U},\Gcal
        \right)$.
For a smooth one-parameter geometry \(\Gcal=\Gcal(\gamma)\), the simple-pole
sensitivity is
\begin{equation*}
\begin{aligned}
        \partial_\gamma\Arec
        &=
        i\,\frac{\Ncal_\gamma+\Ncal_\alpha\,\partial_\gamma\alpha_{KH}}
             {\Dcal_\alpha}
        -
        i\,\frac{\Ncal(\Dcal_{\alpha\gamma}
        +\Dcal_{\alpha\alpha}\partial_\gamma\alpha_{KH})}
             {\Dcal_\alpha^2},\\
        \partial_\gamma\alpha_{KH}
        &=
        -\frac{\Dcal_\gamma}{\Dcal_\alpha},
\end{aligned}
\end{equation*}
with all functions evaluated at
\((\alpha,\omega_{\rm phys},\Gcal)=(\alpha_{KH},\omega_{\rm phys},\Gcal)\).
Thus
$\displaystyle         \partial_\gamma\Arec
        =
        \mathcal S_{\rm force}
        +
        \mathcal S_{\rm pole}
        +
        \mathcal S_{\rm dispersion}$,
where
\[
        \mathcal S_{\rm force}=\frac{i\,\Ncal_\gamma}{\Dcal_\alpha},\qquad
        \mathcal S_{\rm pole}
        =
        i\,\partial_\gamma\alpha_{KH}
        \left(
        \frac{\Ncal_\alpha}{\Dcal_\alpha}
        -
        \frac{\Ncal\Dcal_{\alpha\alpha}}{\Dcal_\alpha^2}
        \right),\qquad
        \mathcal S_{\rm dispersion}
        =
        -\frac{i\,\Ncal\Dcal_{\alpha\gamma}}{\Dcal_\alpha^2}.
\]
The Mellin analogue for finite-angle wedges with self-similar sheet data is
given in Appendix~\ref{app:wedge}.
\subsection*{Limitations and extensions}
The full trailing-edge base flow introduces three analytic issues that are not
settled here.  (i)~\emph{Fredholm realization}: one must construct weighted
spaces for the true unsteady lower-deck operator of
\cite{stewartson1969flow,messiter1970boundary,jobe1974numerical}, prove the index-zero
Fredholm property \eqref{eq:Fredholm-index}--\eqref{eq:adjoint-kernel-one},
and establish the augmented solvability of Hypothesis~\ref{hyp:Fredholm-LD}(ii); the
natural route is the limit-operator decomposition used in
Proposition~\ref{prop:model-fredholm}, with the explicit Airy symbols replaced by the
frozen symbols of the numerically known base
flow~\cite{rabinovitch2004limit}.  (ii)~\emph{Edge nondegeneracy}:
one must prove \(\kappa(\Omega)\neq0\) for the corresponding adjoint
state; Proposition~\ref{prop:model-H4} proves it in the model, where the adjoint is
Airy-explicit.  (iii)~\emph{Spectral degenerations}: multiple wake poles,
pole--cut collisions with the acoustic branch points
\eqref{eq:branch-points}, Briggs--Bers pinches (the absolute/convective
transition, which invalidates the causal downstream pickup of
\Cref{subsec:flat-residue}~\cite{briggs1964electron,bers1983space,monkewitz1990local}),
and the pole strings generated by non-self-similar wedge
sheets~\cite{davis2016instability} all require separate treatment; for a
pole of order \(m\) the residue formula is replaced by the algebraically
growing term
\(\eta_{KH}(x)=i e^{i\alpha_{KH}x}\sum_{\ell=1}^{m}
\tfrac{(ix)^{\ell-1}}{(\ell-1)!}\mathcal R_\ell\).

\appendix
\section{Formal adjoint and concomitant}
\label{app:formal-adjoint}
Let
\[
        \Pi=\mathbb R_X\times\mathbb R_+,\qquad
        \Gamma_p=(-\infty,0)\times\{0\},\qquad
        \Gamma_w=(0,\infty)\times\{0\}.
\]
The steady lower-deck state satisfies
\[
        U_{0X}+V_{0Y}=0,\qquad
        V_0|_{\Gamma_w}=0 .
\]
For \(W=(u,v,p,a)^{\mathsf T}\), define
\[
        R_0(W):=u_X+v_Y ,
\]
\[
        R_1(W):=
        -i\Omega u+U_0u_X+V_0u_Y+U_{0X}u+U_{0Y}v+p_X-u_{YY}.
\]
The primal homogeneous boundary and matching constraints are
\[
        u=v=0\quad\text{on }\Gamma_p,\qquad
        v=0,\quad u_Y=0\quad\text{on }\Gamma_w,
\]
\[
        u(X,Y)-a(X)\to0\quad(Y\to\infty),\qquad
        p=\Kcal[a].
\]
Here
        $\Kcal=H\partial_X$ such that 
$\displaystyle         \widehat{\Kcal f}(\alpha)=|\alpha|\widehat f(\alpha)$.
(For the two-sided wake of \Cref{subsec:two-sided} the computation below is
performed on each half \(\pm Y>0\) and the centerline terms of
\eqref{eq:centerline-conditions} are added; the symmetric component reproduces
exactly the formulas of this appendix, and the antisymmetric component differs
only in the wake-side boundary block.)
\subsection*{Bulk adjoint}
Let \(\Psi=(u^\ast,q)^{\mathsf T}\).  Pair
$\displaystyle        \langle \LTD W,\Psi\rangle
        :=
        \iint_{\Pi}\{u^\ast R_1(W)+qR_0(W)\}\,dX\,dY$ .
Modulo boundary fluxes,
\[
\begin{aligned}
        \iint_\Pi u^\ast U_0u_X
        &\equiv
        -\iint_\Pi (U_0u_X^\ast+U_{0X}u^\ast)u,\\
        \iint_\Pi u^\ast V_0u_Y
        &\equiv
        -\iint_\Pi (V_0u_Y^\ast+V_{0Y}u^\ast)u,\\
        \iint_\Pi u^\ast p_X
        &\equiv
        -\int_{\mathbb R}p\,\bar U_X^\ast\,dX,\qquad
        \bar U^\ast(X):=\int_0^\infty u^\ast(X,Y)\,dY,\\
        -\iint_\Pi u^\ast u_{YY}
        &\equiv
        -\iint_\Pi u_{YY}^\ast u,\qquad
        \iint_\Pi q u_X\equiv-\iint_\Pi q_Xu,\\
        \iint_\Pi qv_Y&\equiv-\iint_\Pi q_Yv .
\end{aligned}
\]
Therefore
\[
\begin{aligned}
        \langle \LTD W,\Psi\rangle
        &=
        \iint_\Pi u\,
        \Big(
        -i\Omega u^\ast-U_0u_X^\ast-U_{0X}u^\ast
        -V_0u_Y^\ast-V_{0Y}u^\ast
        +U_{0X}u^\ast-u_{YY}^\ast-q_X
        \Big)\,dX\,dY\\
        &\quad+
        \iint_\Pi v\,(U_{0Y}u^\ast-q_Y)\,dX\,dY
        -
        \int_{\mathbb R}p\,\bar U_X^\ast\,dX
        +
        \int_{\partial\Pi}J\cdot n\,ds .
\end{aligned}
\]
Using \(V_{0Y}=-U_{0X}\), this becomes
\[
        \langle \LTD W,\Psi\rangle
        =
        \iint_\Pi u\,\mathcal L_u^\ast\Psi\,dX\,dY
        +
        \iint_\Pi v\,\mathcal L_v^\ast\Psi\,dX\,dY
        -
        \int_{\mathbb R}p\,\bar U_X^\ast\,dX
        +
        \int_{\partial\Pi}J\cdot n\,ds ,
\]
where
\[
        \mathcal L_u^\ast\Psi
        =
        -i\Omega u^\ast-U_0u_X^\ast-V_0u_Y^\ast
        +U_{0X}u^\ast-u_{YY}^\ast-q_X,
        \qquad
        \mathcal L_v^\ast\Psi
        =
        U_{0Y}u^\ast-q_Y .
\]
Hence the formal adjoint equations are
\[
        -i\Omega u^\ast-U_0u_X^\ast-V_0u_Y^\ast
        +U_{0X}u^\ast-u_{YY}^\ast-q_X=0,
        \qquad
        q_Y=U_{0Y}u^\ast .
\]
\subsection*{Concomitant}
The boundary fluxes are
\[
        J^X=qu+U_0u^\ast u+u^\ast p,\qquad
        J^Y=qv+V_0u^\ast u-(u^\ast u_Y-u_Y^\ast u).
\]
Equivalently,
\[
        J
        =
        (J^X,J^Y),\qquad
        J^X=qu+U_0u^\ast u+u^\ast p,\quad
        J^Y=qv+V_0u^\ast u-u^\ast u_Y+u_Y^\ast u .
\]
The Lagrange identity is
\[
        u^\ast R_1(W)+qR_0(W)
        -
        u\,\mathcal L_u^\ast\Psi
        -
        v\,\mathcal L_v^\ast\Psi
        =
        \partial_XJ^X+\partial_YJ^Y
        -p\,\bar U_X^\ast\delta_{Y=\infty},
\]
where the last term is understood after integration in \(Y\), equivalently as
the line contribution \(-\int_{\mathbb R}p\,\bar U_X^\ast\,dX\).
\subsection*{Adjoint wall and wake conditions}
On \(\Gamma_p\), the primal variations satisfy \(u=v=0\), while \(u_Y\) is
free.  Hence
$\displaystyle         J^Y|_{\Gamma_p}
        =
        -u^\ast u_Y$,
so vanishing of the boundary form for all admissible \(u_Y\) gives
$u^\ast=0$ on $\Gamma_p$.
On \(\Gamma_w\), the primal variations satisfy \(v=0\), \(u_Y=0\), and
\(V_0=0\).  Hence
$\displaystyle         J^Y|_{\Gamma_w}
        =
        u_Y^\ast u$,
so vanishing for arbitrary admissible \(u\) gives
$u_Y^\ast=0$ on $\Gamma_w$.
Thus
\[
        u^\ast=0\ (X<0,Y=0),\qquad
        u_Y^\ast=0\ (X>0,Y=0),
\]
with decay conditions chosen so that the fluxes at \(X=\pm\infty\) and
\(Y=\infty\) vanish in the weighted graph norm (with the weights
\eqref{eq:weights-new}, adjoint states decay downstream faster than
\(e^{-\vartheta X}\)).
\subsection*{Adjoint of the interaction and matching constraints}
Introduce line multipliers \(b,\mu\) for
\[
        p-\Kcal[a]=0,\qquad
        a-u_\infty=0,\qquad
        u_\infty(X):=\lim_{Y\to\infty}u(X,Y).
\]
The augmented line pairing is
\[
        \mathscr I_{\rm line}
        =
        -\int_{\mathbb R}p\,\bar U_X^\ast\,dX
        +
        \int_{\mathbb R}b(p-\Kcal a)\,dX
        +
        \int_{\mathbb R}\mu(a-u_\infty)\,dX .
\]
The coefficient of \(p\) gives
\[
        -\bar U_X^\ast+b=0
        \quad\Longrightarrow\quad
        b=\bar U_X^\ast.
\]
The coefficient of \(a\) gives
$\displaystyle         \mu-\Kcal^\ast b=0 $.
Since
\[
        H^\ast=-H,\qquad
        \partial_X^\ast=-\partial_X,\qquad
        H\partial_X=\partial_XH,
\]
we have
$\displaystyle         \Kcal^\ast=(H\partial_X)^\ast
        =
        \partial_X^\ast H^\ast
        =
        (-\partial_X)(-H)
        =
        H\partial_X
        =
        \Kcal $.
Equivalently, in Fourier variables,
\[
        \widehat{\Kcal f}(\alpha)=|\alpha|\widehat f(\alpha)
        \quad\Longrightarrow\quad
        \Kcal^\ast=\Kcal\ge0 .
\]
Thus
$\displaystyle         \mu=\Kcal b=\Kcal[\bar U_X^\ast]=H[\bar U_{XX}^\ast]$.
The coefficient of \(u_\infty\) is \(-\mu\), the adjoint far-field traction
associated with the displacement matching.
\subsection*{Compact operator form}
With
\[
        \Psi^\ast=(u^\ast,q,b,\mu)^{\mathsf T},
        \qquad
        \bar U^\ast(X)=\int_0^\infty u^\ast(X,Y)\,dY,
\]
the formal adjoint of the linearized lower-deck operator is
$\displaystyle         \LTD(\Omega)^\ast\Psi^\ast=0 $,
meaning
\[
        \begin{cases}
        -i\Omega u^\ast-U_0u_X^\ast-V_0u_Y^\ast
        +U_{0X}u^\ast-u_{YY}^\ast-q_X=0,\\[1mm]
        q_Y=U_{0Y}u^\ast,\\[1mm]
        b=\bar U_X^\ast,\\[1mm]
        \mu=\Kcal[\bar U_X^\ast],
        \end{cases}
\]
with
\[
        u^\ast=0\quad\text{on }\Gamma_p,\qquad
        u_Y^\ast=0\quad\text{on }\Gamma_w.
\]
For admissible \(W\) and \(\Psi^\ast\),
\begin{equation}
\label{eq:app-Green}
        \langle \LTD(\Omega)W,\Psi^\ast\rangle_{\Hcal_\sigma}
        -
        \langle W,\LTD(\Omega)^\ast\Psi^\ast\rangle_{\Hcal_\sigma}
        =
        \int_{\partial\Pi}J(W,\Psi^\ast)\cdot n\,ds
        +
        \int_{\mathbb R}
        \{b(p-\Kcal a)+\mu(a-u_\infty)\}\,dX .
\end{equation}
If \(W\in\mathcal D(\LTD)\) and \(\Psi^\ast\in\mathcal D(\LTD^\ast)\), all
boundary and line terms vanish except possible finite edge contributions.
\subsection*{Finite edge concomitant}
Let \(B_\rho^+=\Pi\cap\{X^2+Y^2<\rho^2\}\).  For a singular edge datum
\[
        G=C\,r^{-1/2}\Vm(\theta)\in\Eedge,\qquad
        \Eedge=\spn\{r^{-1/2}\Vm\},
\]
with lower-deck line traces \((g_\infty^\sharp,g_K^\sharp)\) given by
\eqref{eq:inner-singular-trace}--\eqref{eq:singular-trace-constants}, define
\[
        \Bedge(G,\Psi^\ast)
        :=
        \fp\lim_{\rho\downarrow0}
        \int_{\partial B_\rho^+}
        J(G,\Psi^\ast)\cdot n\,ds
        +
        \fp\int_{\mathbb R}
        \big(g_K^\sharp\,b+g_\infty^\sharp\,\mu\big)\,dX .
\]
The second (trace-pairing) term is the contribution of the
pressure--displacement and matching constraints, i.e., the line pairing of
\eqref{eq:app-Green} evaluated on the singular datum; it is the form used in
the model computation of \Cref{sec:model}.  The finite parts exist because the
bulk profile is locally square integrable in two dimensions, the line traces
pair against the adjoint line states with local exponents summing above
\(-1\), and the residual divergences are removed by the finite part. 
Linearity yields
$\displaystyle         \Bedge(C\,r^{-1/2}\Vm,\Psi^\ast)
        =
        C\,\Bedge(r^{-1/2}\Vm,\Psi^\ast)$.
Thus the edge nondegeneracy condition used in the main text is exactly
$\displaystyle        \kappa(\Omega)
        :=
        \Bedge(r^{-1/2}\Vm,\Psi^\ast(\Omega))
        \neq0 $.
For the outer family,
$\displaystyle         \Bedge(C_-(A)r^{-1/2}\Vm,\Psi^\ast(\Omega))
        =
        C_-(A)\kappa(\Omega)$.
\subsection*{Resulting compatibility formula}
For decomposed matching data
\[
        F_{\rm match}(A)
        =
        C_-(A)\Fsing+\Finc+A\FKH,
        \qquad
        \Fsing\in\Eedge,\quad
        \Finc+A\FKH\in\Hcal_\sigma,
\]
the Green identity \eqref{eq:app-Green}, applied under the augmented
solvability of Hypothesis~\ref{hyp:Fredholm-LD}(ii), gives the generalized solvability
condition
$\displaystyle         C_-(A)\kappa(\Omega)
        +
        \langle \Finc(\Omega)+A\FKH(\Omega),
        \Psi^\ast(\Omega)\rangle_{\Hcal_\sigma}
        =
        0 $.
Read as an identity in \(A\) (Proposition~\ref{prop:KF-noncirc}), this gives
\(\langle\Finc+A\FKH,\Psi^\ast\rangle=-\kappa(\Omega)C_-(A)\), hence
\(\chi=-\kappa\).  If the lower-deck solution is required to belong to the
bounded graph domain \(\Xcal_\sigma\), then the singular component is excluded,
$\displaystyle         \Pi_{\rm sing}F_{\rm match}(A)=C_-(A)\Fsing=0 $,
and therefore
       $\displaystyle  C_-(A)=0$
iff
       $\displaystyle  \Tssing\nabla\phi_A^{\rm out}=0$
iff
     $\displaystyle    \langle \Finc+A\FKH,\Psi^\ast\rangle_{\Hcal_\sigma}=0 $.
Consequently,
$\displaystyle         A
        =
        -\frac{
        \langle \Finc(\Omega),\Psi^\ast(\Omega)\rangle_{\Hcal_\sigma}}
        {
        \langle \FKH(\Omega),\Psi^\ast(\Omega)\rangle_{\Hcal_\sigma}}
        =
        -\frac{C_-^{(0)}}{C_-^{(KH)}} $.

% =====================================================================
\section{Mellin analogue for finite-angle wedges}
\label{app:wedge}
Let the edge be a wedge
\[
        \Pi_\Theta=\{(r,\theta):r>0,\ -\Theta_-<\theta<\Theta_+\},
        \qquad
        \Theta=\Theta_-+\Theta_+ .
\]
For local wedge fields use the Mellin transform
\[
        \mathfrak M[f](s)
        =
        \int_0^\infty r^{s-1}f(r)\,dr,\qquad
        f(r)=\frac1{2\pi i}\int_{\Re s=\sigma}
        r^{-s}\mathfrak M[f](s)\,ds .
\]
A structural caveat is required which has no flat-plate counterpart.  The
Mellin transform diagonalizes dilations, not translations; a wake mode
that is asymptotically a plane wave \(e^{i\alpha x}\) far from the edge is
not Mellin-homogeneous, and for general wedge sheets the instability emerges
from infinite pole strings \(\{s_{KH}-n\}_{n\geq0}\) generated by the
functional-difference structure rather than from a single pole---this is
precisely the careful treatment that the instability-wave amplitude requires
in~\cite{davis2016instability}.  A single-residue statement is therefore
meaningful only on the dilation-invariant class:
\begin{assumption}[Self-similar sheet class]
\label{ass:wedge-selfsimilar}
The base sheet data in \(\Gcal\) are self-similar (sheet strength a pure
power of \(r\)), so that the homogeneous wake modes of the wedge problem are
Mellin-homogeneous, \(\eta_{KH}\propto r^{-s_{KH}}\), the wake functional
\(\ell_{KH}\) is the coefficient of \(r^{-s_{KH}}\), and the Kutta-normalized
response \eqref{eq:wedge-M} is meromorphic near a simple wake pole
\(s_{KH}\). 
\end{assumption}
The principal separated solutions are
\[
        \phi(r,\theta)=r^\lambda\Phi_\lambda(\theta),\qquad
        \Phi_\lambda''+\lambda^2\Phi_\lambda=0,
\]
or, in Mellin notation, \(\lambda=-s\).  The wedge walls and sheet conditions
produce a finite-dimensional functional-difference
system~\cite{lawrie1994exact,davis2016instability}
\begin{equation}
\label{eq:Mellin-difference}
        \mathbb A(s;\Omega,\Gcal)\mathbf U(s)
        +
        \mathbb B(s;\Omega,\Gcal)\mathbf U(s-1)
        =
        \mathbf F(s;\Omega,\Gcal),
\end{equation}
where \(\mathbf U(s)\) collects Mellin transforms of the angular trace
amplitudes and sheet displacement.  Its homogeneous determinant is
$\displaystyle         \Dcal_{\rm wedge}(s;\Omega,\Gcal)
        :=
        \det\mathbb T(s;\Omega,\Gcal)$,
after reduction of the difference system to a period-one transfer matrix
\(\mathbb T\).  The Kutta-normalized wedge response is assumed meromorphic:
\begin{equation}
\label{eq:wedge-M}
        \Mcal_{\rm wedge}(s;\Omega,\Gcal)
        =
        \frac{\Ncal_{\rm wedge}(s;\Omega,\Gcal)}
             {\Dcal_{\rm wedge}(s;\Omega,\Gcal)} .
\end{equation}
The wedge faces carry the rigid Neumann conditions
\(\partial_\theta\phi=0\) at \(\theta=-\Theta_-,\Theta_+\), so the angular
pencil \(\Phi_\lambda''+\lambda^2\Phi_\lambda=0\) with
\(\Phi_\lambda'(-\Theta_-)=\Phi_\lambda'(\Theta_+)=0\) has the Neumann spectrum
\[
        \lambda_n(\Theta)=\frac{n\pi}{\Theta},\qquad n=0,1,2,\ldots,\qquad
        \Theta=\Theta_-+\Theta_+,
\]
with eigenfunctions
\(\Phi_{\lambda_n}(\theta)=\cos\!\big(\tfrac{n\pi}{\Theta}(\theta+\Theta_-)\big)\).
The first nonconstant root is the loading (pressure-jump) mode
\[
        \lambda_-(\Theta)=\frac{\pi}{\Theta},\qquad
        \Psi_-^{(\Theta)}(\theta)
        =
        \cos\!\Big(\frac{\pi}{\Theta}(\theta+\Theta_-)\Big),
\]
which at the cusped/flat-plate value \(\Theta=2\pi\) reduces to
\(\lambda_-=\tfrac12\) and \(\Psi_-^{(\Theta)}=-\sin\tfrac{\theta}{2}\),
recovering Assumption~\ref{ass:edge-indicial}.  The associated velocity singularity is
\(r^{\lambda_-(\Theta)-1}\), genuinely singular precisely when
\(\lambda_-(\Theta)<1\), i.e. \(\Theta>\pi\); the admissible trailing-edge
range is therefore \(\Theta\in(\pi,2\pi]\), the singularity (and with it the
Kutta selection) disappearing as \(\Theta\downarrow\pi\).  The wedge singular
edge-trace space is
$\displaystyle         \Eedge^{(\Theta)}
        =
        \spn\{r^{\lambda_-(\Theta)-1}\Vm^{(\Theta)}\}$,
replacing \(\Eedge=\spn\{r^{-1/2}\Vm\}\) in the lower-deck hypotheses.
The edge singularity \(r^{\lambda_-(\Theta)}\Psi_-^{(\Theta)}(\theta)\)
corresponds to a Mellin pole at
$\displaystyle         s=-\lambda_-(\Theta)=-\frac{\pi}{\Theta}$
under the convention \(\mathfrak M[r^\lambda]\sim(s+\lambda)^{-1}\).  Thus the
wedge Kutta normalization is
\[
        \Res_{s=-\pi/\Theta}\Mcal_{\rm wedge}(s;\Omega,\Gcal)=0,
        \qquad
        C_-(A)=0 ,
\]
which collapses to \(\Res_{s=-1/2}=0\) only in the cusped limit
\(\Theta=2\pi\).
\begin{remark}[Separation of edge and wake poles]
\label{rem:wedge-separation}
For \(\Mcal_{\rm wedge}\) to be well defined near the wake pole \(s_{KH}\), the
edge pole and the wake pole must remain distinct, \(s_{KH}\neq-\pi/\Theta\): the
normalization kills the edge pole and the residue reads off the wake pole.  For
thin trailing edges \(\Theta\to2\pi\) this is automatic; as \(\Theta\to\pi^+\)
the edge pole \(-\pi/\Theta\to-1\) migrates and may in principle collide with a
wake pole.
\end{remark}
Let \(s_{KH}\) be a simple unstable wake pole:
\begin{equation}
\label{eq:wedge-KH-pole}
        \Dcal_{\rm wedge}(s_{KH};\Omega,\Gcal)=0,\qquad
        \partial_s\Dcal_{\rm wedge}(s_{KH};\Omega,\Gcal)\neq0 .
\end{equation}
Then the inverse Mellin deformation gives
\[
        \eta(r)
        =
        \frac1{2\pi i}\int_{\Re s=\sigma}
        r^{-s}\Mcal_{\rm wedge}(s;\Omega,\Gcal)\,ds
        =
        \sum_{s_j\in\mathcal P}
        \Res_{s=s_j}\big(r^{-s}\Mcal_{\rm wedge}(s)\big)
        +\eta_{\rm rem}(r),
\]
the orientation factor being unity here
(\(\tfrac1{2\pi i}\cdot2\pi i=1\), in contrast with the Fourier factor \(i\)
of \eqref{eq:contour-deformation}), and the unstable wedge-mode coefficient is
$\displaystyle         A_{\rm pole}^{\rm wedge}(\Omega,\Gcal)
        =
        \Res_{s=s_{KH}}\Mcal_{\rm wedge}(s;\Omega,\Gcal)$.
For a simple pole,
$\displaystyle         A_{\rm pole}^{\rm wedge}(\Omega,\Gcal)
        =
        \frac{\Ncal_{\rm wedge}(s_{KH};\Omega,\Gcal)}
             {\partial_s\Dcal_{\rm wedge}(s_{KH};\Omega,\Gcal)} $.
\begin{theorem}[Mellin pole-residue formula]
\label{thm:wedge-pole-residue}
Assume Assumption~\ref{ass:wedge-selfsimilar}, the wedge Mellin representation
\eqref{eq:wedge-M}, the simple pole condition \eqref{eq:wedge-KH-pole}, and the
lower-deck hypotheses Hypotheses~\ref{hyp:Fredholm-LD} and~\ref{hyp:edge-concomitant}.  Then
\begin{equation}
\label{eq:wedge-main}
        \Arec^{\rm wedge}(\Omega,\Gcal)
        =
        -
        \frac{
        \langle \Finc^{\rm wedge}(\Omega),\Psi^\ast_{\rm wedge}(\Omega)
        \rangle_{\Hcal_\sigma}}
        {
        \langle \FKH^{\rm wedge}(\Omega),\Psi^\ast_{\rm wedge}(\Omega)
        \rangle_{\Hcal_\sigma}}
        =
        \Res_{s=s_{KH}}
        \Mcal_{\rm wedge}(s;\Omega,\Gcal).
\end{equation}
If \(s_{KH}\) is simple, then
$\displaystyle         \Arec^{\rm wedge}(\Omega,\Gcal)
        =
        \frac{\Ncal_{\rm wedge}(s_{KH};\Omega,\Gcal)}
             {\partial_s\Dcal_{\rm wedge}(s_{KH};\Omega,\Gcal)}$.
\end{theorem}
\begin{proof}
The Mellin transform diagonalizes the radial homogeneity and, under
Assumption~\ref{ass:wedge-selfsimilar}, converts the wedge trace equations into
\eqref{eq:Mellin-difference} with a Mellin-homogeneous wake mode.  Kutta
normalization removes the \(s=-\pi/\Theta\) edge pole, hence fixes the same
one-dimensional kernel as \(C_-(A)=0\).  
The conditional lower-deck theorem
identifies this Kutta-normalized wedge solution with the Fredholm-selected
one.  The inverse Mellin formula then gives the coefficient of the unstable
wedge mode as the residue at \(s=s_{KH}\), which yields \eqref{eq:wedge-main};
the simple-pole formula follows by the Laurent expansion of
\(\Dcal_{\rm wedge}\).
\end{proof}
Together with \Cref{thm:conditional-viscous-Kutta}, the wedge analogue of
\eqref{eq:full-final-identity} holds verbatim, with
\(i\Res_{\alpha=\alpha_{KH}}\Mcal\) replaced by
\(\Res_{s=s_{KH}}\Mcal_{\rm wedge}\) and the edge space \(\Eedge\) by
\(\Eedge^{(\Theta)}\).

\printcredits

\bibliographystyle{cas-model2-names}
\bibliography{reference}

\end{document}